\documentclass[11pt,twoside,reqno]{amsart}

\usepackage{microtype}
\usepackage{cite}
\usepackage[OT1]{fontenc}
\usepackage{type1cm}
\usepackage{amssymb}
\usepackage{comment}
\usepackage{xcolor}
\usepackage{dsfont}

\usepackage{geometry}
\usepackage{hyperref}
\hypersetup{
  colorlinks=true,
  linkcolor=black,
  anchorcolor=black,
  citecolor=black,
  filecolor=black,      
  menucolor=red,
  runcolor=black,
  urlcolor=black,
}

\numberwithin{equation}{section}

\theoremstyle{plain}
\newtheorem{theorem}{Theorem}[section]

\newtheorem{corollary}[theorem]{Corollary}
\newtheorem{proposition}[theorem]{Proposition}
\newtheorem*{claim}{Claim}
\newtheorem{lemma}[theorem]{Lemma}

\theoremstyle{remark}

\theoremstyle{definition}
\newtheorem{definition}[theorem]{Definition}

\newcommand{\R}{\mathbb{R}}

\newcommand{\C}{\mathbb{C}}
\newcommand{\Q}{\mathbb{Q}}
\newcommand{\Z}{\mathbb{Z}}
\newcommand{\N}{\mathbb{N}}

\newcommand{\iii}{\mathtt{i}}

\newcommand{\Wedge}{\textstyle\bigwedge}

\renewcommand{\geq}{\geqslant}
\renewcommand{\leq}{\leqslant}

\DeclareMathOperator*{\essinf}{ess\,inf}

\DeclareMathOperator{\dimh}{dim_H}

\DeclareMathOperator{\diam}{diam}

\DeclareMathOperator{\spt}{spt}

\allowdisplaybreaks

\begin{document}

\title[dissonance and orthogonal projections]{On dissonance and orthogonal projections of self-conformal measures}

\author{Aleksi Py\"or\"al\"a}
\address[Aleksi Py\"or\"al\"a]
       {DEPARTMENT OF MATHEMATICS AND STATISTICS \\ 
        P.O. BOX 35 (MAD) \\ 
        FI-40014 UNIVERSITY OF JYVÄSKYLÄ \\ 
        Finland}
\email{aleksi.pyorala@gmail.com}

\thanks{The author was supported by the Research Council of Finland via the projects \emph{Approximate incidence geometry}, grant no. 355453, and \emph{Fractals and randomness}, grant no. 368817.}
\date{\today}

\begin{abstract}
Let $\mu$ be a self-conformal measure on $\R^d$. We establish conditions for $\mu$ under which 
\begin{equation}\label{eq-dissonance}
    \dim(\mu*\nu) = \min\lbrace d,\dim\mu+\dim\nu\rbrace
\end{equation}
holds when $\nu$ is any Ahlfors-regular or self-conformal measure on $\R^d$. Our main result states that \eqref{eq-dissonance} holds if $\mu$ is totally non-linear and not supported on a smooth hypersurface. We also establish (likely non-sharp) algebraic conditions for $\mu$ under which \eqref{eq-dissonance} holds without assuming non-linearity. In addition, we show that 
\begin{equation*}
    \dim \mu\circ\pi^{-1} = \min\{ k, \dim \mu\}
\end{equation*}
for every ortohogonal projection $\pi:\R^d\to\R^k$, $0<k<d$, when either $d=2$ and $\mu$ is not self-similar and not supported on a line, or $d\geq 3$ and $\mu$ is totally non-linear and not supported on a smooth hypersurface.
\end{abstract}

\maketitle

\section{Introduction}

Two exact-dimensional Borel probability measures $\mu$ and $\nu$ on $\R^d$ are said to \emph{dissonate} if 
\begin{equation}\label{eq-dissonancedefinition}
    \dim(\mu*\nu) = \min\lbrace d,\dim\mu+\dim\nu\rbrace.
\end{equation}
Here and throughout, we use $\dim$ to denote the Hausdorff dimension of both sets and measures: 
\begin{align*}
    \dim \mu = \essinf_{x\sim \mu} \liminf_{r\to 0}\frac{\log \mu(B(x,r))}{\log r} = \inf \lbrace \dim A:\ \mu(A) > 0\rbrace
\end{align*}
for any Borel measure $\mu$. If \eqref{eq-dissonancedefinition} fails, $\mu$ and $\nu$ are said to \emph{resonate}. For example, any exact-dimensional measures $\mu$ and $\nu$ dissonate after randomly scaling and rotating either of them, as was shown by López and Moreira \cite[Theorem 2.3]{LopezMoreira2015}. Because of this, resonance between measures is in some sense an exceptional phenomenon, and might be interpreted as evidence of common arithmetic structure shared by both measures.

The absence of arithmetic structure in self-conformal measures has been an active research topic in the past years, not only in the form of their dissonance but also in the rate of their Fourier decay, $L^2$-flattening behavior and the prevalence of normal numbers in their support, see \cite{AlgomChangWuWu2025, AlgomKhalil2025, AlgomRHWang2023, AlgomRHWang2024, RossiShmerkin2020, BakerBanaji2025, BakerSahlsten2023, AlgomOrponen2025, BakerKhalilSahlsten2024} for a non-exhaustive list of recent advances in these connected topics. To very shortly illustrate the connection to dissonance, note that Salem measures, i.e. measures with fastest possible Fourier decay, dissonate with \emph{every} exact-dimensional Borel measure, and the dimension of iterated convolutions of $L^2$-flattening measures on $\R^d$ approaches $d$. It is believed that in general, a self-conformal measure should be ``highly non-arithmetic'' in every such aspect, at least in the presence of suitable non-linearity: Their Fourier transforms should have fast decay, they should be $L^2$-flattening and their supports should contain many normal numbers. In particular, self-conformal measures should dissonate, possibly with \emph{any other} exact-dimensional measure.

We study dissonance between probability measures $\mu$ and $\nu$ on $\R^d$ when one of them is self-conformal. We will next recall the definition of self-conformal measures.

Let $\Gamma$ be a finite set with $\#\Gamma\geq 2$ and let $\Phi = \lbrace f_i\rbrace_{i\in\Gamma}$ be an iterated function system of contractive conformal maps on $(-1,1)^d$, see Section 2 for the definitions. A Borel probability measure $\mu$ on $(-1,1)^d$ is called \emph{self-conformal} if there exist numbers $0<p_i<1$ such that $\sum_{i\in\Gamma} p_i = 1$ and
\begin{equation}\label{eq-selfconformalmeasure}
    \mu = \sum_{i\in\Gamma} p_i \cdot f_i\mu,
\end{equation}
where we write $f\mu := \mu\circ f^{-1}$. A conformal iterated function system $\Phi$ is \emph{totally non-linear} if for every $g\in\mathcal C_{\rm conf}((-1,1)^d)$ there exists $f \in \Phi$ such that $g \circ f \circ g^{-1}$ is not an affine function. 

In general, a self-conformal measure need not dissonate even with itself. However, the only known examples of resonating self-conformal measures arise from iterated function systems where the set of \emph{rotations} $\lbrace Df(x)\Vert Df(x)\Vert^{-1}:\ f\in \Phi, x\in (-1,1)^d\rbrace$ is contained in a strict subset of $SO(d)$. Prohibiting such a restriction, for example by assuming that $\Phi$ is totally non-linear and does not preserve a hypersurface of $\R^d$, it seems at least possible that self-conformal measures associated to $\Phi$ would dissonate with any other exact-dimensional measure. While dissonance with arbitrary measures is far beyond the scope of the present work, we manage to establish that many self-conformal measures dissonate with a class of measures which we call \emph{tangent-regular}, a class which contains all Ahlfors-regular and self-conformal probability measures. 
\begin{definition}
    A Borel probability measure $\nu$ on $\R^d$ is \emph{tangent-regular} if the following holds for $\nu$-almost every $x\in \R^d$: If $P$ is any tangent distribution of $\nu$ at $x$, then $\dim \eta = \dim \nu$ for $P$-almost every $\eta$.  
\end{definition}
For the definition of tangent distributions, see Section 2. For example, Ahlfors-regular and self-conformal measures are tangent-regular. See Lemma \ref{lemma-tangentregular} for a proof of the latter claim. 

\begin{theorem}\label{thm-1}
    Let $\Phi$ be a totally non-linear conformal iterated function system on $\R^d$, and let $\mu$ be a self-conformal measure associated to $\Phi$. Suppose that either 
    \begin{enumerate}
        \item $d=1$ and $\Phi$ is real-analytic, 
        \item $d=2$ and $\mu$ is not supported on an analytic curve, or
        \item $d\geq3$ and $\mu$ is not supported on an affine hyperplane or a sphere.
    \end{enumerate}
    Then for any tangent-regular measure $\nu$ on $\R^d$, 
    \begin{equation*}
        \dim(\mu*\nu) = \min\lbrace d, \dim\mu+\dim\nu\rbrace.
    \end{equation*}
\end{theorem}

One can also define dissonance between sets by replacing in \eqref{eq-dissonancedefinition} the convolution $\mu*\nu$ by the sumset $X+Y = \lbrace x+y:\ x\in X, Y\in Y\rbrace$ for any $X,Y\subset \R^d$. In particular, one may ask whether a \emph{self-conformal set} dissonates with a given Borel set in $\R^d$. Recall that given a conformal iterated function system $\Phi$, there exists a unique compact set $K$, called the attractor of $\Phi$, which satisfies
\begin{equation}\label{eq-selfconformalset}
    K = \bigcup_{i\in\Gamma} f_i(K).
\end{equation}
The set $K$ is also called the self-conformal set associated to $\Phi$ and for any self-conformal measure $\mu$ associated to $\Phi$ it holds that $K = {\rm spt}(\mu)$. The family $\Phi$ satisfies the \emph{strong separation property} if $f_i([-1,1]^d)$ are disjoint for every $i\in \Gamma$. 

Since self-conformal sets support self-conformal measures with dimension arbitrarily close to that of the set (see Lemma \ref{lemma-supportsmeasure}), Theorem \ref{thm-1} as well as the variants below imply analogous results for sets.

\begin{corollary}
    Let $\Phi$ be a totally non-linear conformal iterated function system on $\R^d$, and let $X$ denote the attractor of $\Phi$. Suppose that either 
    \begin{enumerate}
        \item $d=1$ and $\Phi$ is real-analytic, 
        \item $d=2$ and $X$ is not contained in an analytic curve, or
        \item $d\geq3$ and $X$ is not contained in an affine hyperplane or a sphere.
    \end{enumerate} 
    Then for any Ahlfors-regular or self-conformal set $Y\subset\R^d$,
    \begin{equation*}
        \dim(X+Y) = \min\lbrace d, \dim X+\dim Y\rbrace.
    \end{equation*}
\end{corollary}

We believe the conclusion of Theorem \ref{thm-1} should hold with only the assumption that the totally non-linear self-conformal measure $\mu$ is not supported in an affine hyperplane. 

We also obtain the following projection theorem for totally non-linear self-conformal measures. 

\begin{theorem}\label{thm-8}
    Let $\Phi$ be a totally non-linear conformal iterated function system on $\R^d$, and let $\mu$ be a self-conformal measure associated to $\Phi$. Suppose that $\mu$ is not supported on an affine hyperplane or a sphere. Then for any $0<k<d$ and any orthogonal projection $\pi:\R^d\to\R^k$, 
    \begin{equation*}
        \dim \mu \circ \pi^{-1}= \min\lbrace k, \dim\mu\rbrace.
    \end{equation*}
\end{theorem}

\begin{corollary}
    Let $\Phi$ be a totally non-linear conformal iterated function system on $\R^d$ with attractor $X$. Suppose that $X$ is not contained in an affine hyperplane or a sphere. Then for any $0<k<d$ and any orthogonal projection $\pi:\R^d\to\R^k$, 
    \begin{equation*}
        \dim \pi(X) = \min\lbrace k, \dim X\rbrace.
    \end{equation*}
\end{corollary}

\subsection{Self-conformal measures on the plane}

In the case $d=2$, we obtain sharper versions of Theorem \ref{thm-1} when convolving two self-conformal measures, and Theorem \ref{thm-8} when projecting any self-conformal measure.

\begin{theorem}\label{thm-2}
    Let $\Phi$ and $\Psi$ be totally non-linear conformal iterated function systems on $\R^2$, and let $\mu$ and $\nu$ be self-conformal measures associated to $\Phi$ and $\Psi$. Suppose that $\mu$ and $\nu$ are not supported on parallel lines. Then
    \begin{equation*}
        \dim(\mu*\nu) = \min\lbrace 2, \dim\mu+\dim\nu\rbrace.
    \end{equation*}
\end{theorem}

The assumption that $\mu$ and $\nu$ are not supported on parallel lines is necessary. However, it would likely suffice to assume that only $\Phi$ is totally non-linear. 

\begin{corollary}\label{cor-2}
    Let $\Phi$ and $\Psi$ be totally non-linear conformal iterated function systems on $\R^2$, and let $X$ and $Y$ denote their attractors. Suppose that $X$ and $Y$ are not contained in parallel lines. Then
    \begin{equation*}
        \dim(X+Y) = \min\lbrace 2, \dim X+\dim Y\rbrace.
    \end{equation*}
\end{corollary}

By combining Theorem \ref{thm-8} with known exceptional set estimates for self-\emph{similar} measures on the plane, we are able to conclude that self-conformal measures which are not self-similar have no dimension-reducing projections.

\begin{theorem}\label{thm-plane}
    Let $\Phi$ be a conformal iterated function system on $\R^2$, and let $\mu$ be a self-conformal measure associated to $\Phi$. Suppose that $\Phi$ contains a map which is not a similarity, and $\mu$ is not supported on a line. Then for any orthogonal projection $\pi:\R^2\to\R$, 
    \begin{equation*}
        \dim \mu \circ \pi^{-1}= \min\lbrace 1, \dim\mu\rbrace.
    \end{equation*}
\end{theorem}

\begin{corollary}
    Let $\Phi$ be a totally non-linear conformal iterated function system on $\R^2$ with attractor $K$. Suppose that $\Phi$ contains a map which is not a similarity, and $K$ is not contained on a line. Then for any orthogonal projection $\pi:\R^2\to\R$,
    \begin{equation*}
        \dim \pi(X) = \min\lbrace 1, \dim X\rbrace.
    \end{equation*}
\end{corollary}

\subsection{History}

Let us then review what was know of dissonance and orthogonal projections of self-conformal measures prior to Theorems \ref{thm-1}, \ref{thm-8}, \ref{thm-2} and \ref{thm-plane}. 

First, the linear case $d=1$ of Theorem \ref{thm-1} is known, or at least follows easily from existing results. Namely, in \cite[Theorem 1.5]{BaranyKaenmakiPyoralaWu2023} it was shown that if $\mu$ and $\nu$ are self-conformal measures on $\R$, then \eqref{eq-dissonancedefinition} holds under an algebraic assumption on the defining interated function systems. While the result for totally non-linear self-conformal measures without algebraic assumptions is not explicitly stated, it follows with a very similar argument. Nevertheless, we will include a short argument covering this case in the proof of Theorem \ref{thm-1}. 

In the plane, Moreira and Zamudio \cite[Theorem B]{MoreiraZamudio2024} have shown that the conclusion of Corollary \ref{cor-2} holds for a class of self-conformal sets in $\C$ with the strong separation condition and an algebraic assumption on the defining iterated function systems. There the authors also considers sums of more than just two self-conformal sets. For a similar but non-intersecting class of \emph{self-affine sets} a result analogous to Corollary \ref{cor-2} was obtained in \cite{Pyorala2024}, also under an algebraic assumption on the defining iterated function systems. In dimensions $d\geq 3$ we are not aware of any existing results on dissonance of self-conformal sets or measures. 

Regarding orthogonal projections, Bruce and Jin \cite{BruceJin2019} study orthogonal projections of self-conformal measures on $\R^d$ which satisfy an additional dynamical assumption. It does not seem clear whether this assumption is satisfied by totally non-linear self-conformal measures. Our proof of Theorem \ref{thm-8} differs from the approach of \cite{BruceJin2019}. A version of Theorem \ref{thm-projectionsalgebraic} for self-similar measures was proven by Falconer and Jin \cite{FalconerJin2014} in 2014 under the same assumption \eqref{eq-bestcondition2}. Recently, Algom and Shmerkin \cite{AlgomShmerkin2025} managed to substantially weaken the assumption \eqref{eq-bestcondition2} for self-similar measures by using recent restricted projection theorems. 



\subsection{Self-conformal measures with additional algebraic assumptions}
It is possible to replace the assumptions of non-linearity in Theorems \ref{thm-1} and \ref{thm-8} with various algebraic conditions on the defining iterated function system.

We let $CO(d)$ denote the conformal (special) orthogonal group of $\R^d$, 
\begin{equation*}
    CO(d) = \lbrace r O:\ r>0, O\in SO(d)\rbrace.
\end{equation*}
In particular, if $f$ is an orientation-preserving conformal map on $U$, then $Df(x) \in CO(d)$ for every $x\in U$. For a contractive conformal map $f$ with unique fixed point $x_0$, we write 
\begin{align*}
    \lambda(f) &:= \Vert Df(x_0)\Vert \in \R, \\
    O(f) &:= Df(x_0)\Vert Df(x_0)\Vert^{-1} \in O(d).
\end{align*}
If $G$ is a topological group and $\mathcal A\subset G$, we write
\begin{equation*}
    \langle \mathcal A\rangle := \overline{\lbrace A_1 A_2\cdots A_n:\ n\geq 1, A_i\in \mathcal A\rbrace}.
\end{equation*}

\begin{theorem}\label{thm-3}
    Let $\Phi$ be a conformal iterated function system on $\R^d$ and let $\mu$ be a self-conformal measure associated to $\Phi$. Suppose that 
    \begin{equation}\label{eq-bestcondition}
    \langle \lbrace \lambda(f) O(f), \lambda(f)^{-1} O(f)^{-1}:\ f\in \Phi\rbrace\rangle \supseteq CO(d).
    \end{equation}
    Then for any tangent-regular measure $\nu$ on $\R^d$, 
    \begin{equation*}
        \dim(\mu*\nu) = \min\lbrace d, \dim\mu+\dim\nu\rbrace.
    \end{equation*}
\end{theorem}

The condition \eqref{eq-bestcondition} in Theorem \ref{thm-3} is the most general one the present method allows, but is likely stronger than what is needed for the conclusion. In particular, it seems possible that the assumption $\langle \lbrace O(f):\ f\in \Phi\rbrace\rangle \supseteq SO(d)$ could be enough. At least we are not aware of any counterexamples.  

Regarding orthogonal projections we have the following.
\begin{theorem}\label{thm-projectionsalgebraic}
    Let $\Phi$ be a conformal iterated function system on $\R^d$ and let $\mu$ be a self-conformal measure associated to $\Phi$. Suppose that 
    \begin{equation}\label{eq-bestcondition2}
    \langle \lbrace O(f):\ f\in \Phi\rbrace\rangle \supseteq SO(d).
    \end{equation}
    Then for any $0<k<d$ and any orthogonal projection $\pi:\R^d\to\R^k$, 
    \begin{equation*}
        \dim \mu\circ\pi^{-1} = \min\lbrace k,\dim\mu\rbrace.
    \end{equation*}
\end{theorem}

We will next review some algebraic assumpions on $\Phi$ which ensure that the condition \eqref{eq-bestcondition} of Theorem \ref{thm-3} is satisfied. Let 
\begin{equation*}
    O(\Phi) := \lbrace O(f), O(f)^{-1}:\ f\in \Phi\rbrace.
\end{equation*}

\begin{corollary}\label{thm-4}
    Let $\Phi$ be a conformal iterated function system on $\R^d$ with $d\geq 3$, and let $\mu$ be a self-conformal measure associated to $\Phi$. Suppose that 
    \begin{enumerate}
        \item $\langle O(\Phi)\rangle \supseteq SO(d)$,
        \item there exist $f,g \in\Phi$ such that $\frac{\log \lambda(f)}{\log\lambda(g)}\not\in\Q$, and
        \item there exist $f,g\in \Phi$ with $O(f) = O(g)$ and $\lambda(f)\neq \lambda(g)$. 
    \end{enumerate}
    Then for any tangent-regular measure $\nu$ on $\R^d$, 
    \begin{equation*}
        \dim(\mu*\nu) = \min\lbrace d, \dim\mu+\dim\nu\rbrace.
    \end{equation*}
\end{corollary}

The conditions (1) - (3) together ensure that \eqref{eq-bestcondition} holds so that we may apply Theorem \ref{thm-3}. While one might expect that (1) and (2) (at least in some form) are necessary for \eqref{eq-bestcondition}, we do not know whether the assumption (3) is necessary. In the plane we have the following.

\begin{corollary}\label{thm-5}
   Let $\Phi$ be a conformal iterated function system on $\R^2$, and let $\mu$ be a self-conformal measure associated to $\Phi$. Suppose that 
    \begin{enumerate}
        \item $\langle O(\Phi)\rangle \supseteq  SO(2)$,
        \item there exist $f_1, f_2, f_3, f_4\in \Phi$ (not necessarily distinct) such that $O(f_1) = O(f_2)$, $O(f_3) = O(f_4)$ and $\frac{\log \lambda(f_1)-\log\lambda(f_2)}{\log\lambda(f_3)-\log\lambda(f_4)}\not\in\Q$. 
    \end{enumerate}
    Then for any tangent-regular measure $\nu$ on $\R^2$, 
    \begin{equation*}
        \dim(\mu*\nu) = \min\lbrace 2, \dim\mu+\dim\nu\rbrace.
    \end{equation*}
\end{corollary}

Again, the conditions (1) and (2) ensure that \eqref{eq-bestcondition} holds. In the $d=2$ case, however, both (1) and (2) are necessary for this: Since $SO(2) \cong \R/\Z$, one can choose affine maps $f_i(x) = r_i O_i x + a_i$ in such a way that $\langle \lbrace O_i\rbrace \rangle = SO(2)$ and $\frac{\log r_i}{\log r_j} \not\in \Q$ for every $i\neq j$, but $\langle \lbrace r_i O_i:\ i\in\Gamma\rbrace\rangle$ is contained in a one-dimensional submanifold of $CO(2)$, which prevents \eqref{eq-bestcondition}.

Ultimately, the proofs of both Corollaries \ref{thm-4} and \ref{thm-5} rely on Theorem \ref{thm-3}, and it seems at least possible that just the assumption (1) could be enough for both conclusions.  

\subsection{Dissonance of self-conformal measures}

For convolutions between two self-conformal measures, we have very similar algebraic conditions. For $A,B\subset CO(d)$, write $A^{-1} := \lbrace a^{-1}:\ a\in A\rbrace$ and $AB := \lbrace ab:\ a\in A, b\in B\rbrace$. 

\begin{theorem}\label{thm-6}
    Let $\Phi$ and $\Psi$ be conformal iterated function systems on $\R^d$ with $d\geq 3$, and let $\mu$ and $\nu$ be self-conformal measures associated to $\Phi$ and $\Psi$. Suppose that 
    \begin{enumerate}
        \item $\langle O(\Phi) O(\Psi)^{-1} \rangle \supseteq  CO(d)$, 
        \item there exist $f, g\in\Phi\cup \Psi$ such that $\frac{\log\lambda(f)}{\log\lambda(g)}\not\in\Q$, and 
        \item there exist $f,g \in \Phi \cup \Psi$ such that $O(f) = O(g)$ and $\lambda(f)\neq \lambda(g)$. 
    \end{enumerate}
    Then
    \begin{equation*}
        \dim(\mu*\nu) = \min\lbrace d, \dim\mu+\dim\nu\rbrace.
    \end{equation*}
\end{theorem}

\begin{theorem}\label{thm-7}
    Let $\Phi$ and $\Psi$ be conformal iterated function systems on $\R^2$, and let $\mu$ and $\nu$ be self-conformal measures associated to $\Phi$ and $\Psi$. Suppose that 
    \begin{enumerate}
        \item $\langle O(\Phi) O(\Psi)^{-1}\rangle \supseteq  CO(2)$, and 
        \item there exist $f_1, f_2, f_3,f_4\in \Phi\cup\Psi$ (not necessarily distinct) with $O(f_1) = O(f_2)$, $O(f_3) = O(f_4)$ and $\frac{\log\lambda(f_1)-\log\lambda(f_2)}{\log\lambda(f_3)-\log\lambda(f_4)}\not\in\Q$. 
    \end{enumerate}
    Then
    \begin{equation*}
        \dim(\mu*\nu) = \min\lbrace 2, \dim\mu+\dim\nu\rbrace.
    \end{equation*}
\end{theorem}

\subsection{Proof strategy} 
Our proofs are heavily inspired by Hochman's proof of \cite[Theorem 1.39]{Hochmanpreprint} which shows Theorem \ref{thm-1} for convolutions between self-similar and Ahlfors-regular measures on the line under the strong separation condition and an algebraic assumption on the defining iterated function system. The proof of Theorem \ref{thm-1} consists of the following three steps. Let $\mu$ and $\nu$ be as in Theorem \ref{thm-1}.

{\bf Step 1.} Using the fact that $\mu$ is self-conformal we show that the \emph{tangent distributions} (see Section 2) of $\mu\times\nu$ have certain invariance under linear maps in the first $d$ coordinates. Specifically, if $P$ is (an ergodic component of) a tangent distribution of $\mu\times\nu$, then 
\begin{equation*}
     P\circ (L, I)^{-1} = P
\end{equation*}
for a rich class of conformal linear maps $L:\R^d\to\R^d$, where $(L,I)$ denotes the map $\mu\times\nu \mapsto (\mu\circ L^{-1}) \times \nu$. This invariance is recorded in Proposition \ref{prop-tangentdistributioninvariant}, and its proof very closely follows the proof of \cite[Theorem 1.1]{BaranyKaenmakiPyoralaWu2023}.

{\bf Step 2.} Relying on the assumption of non-linearity (or, in later theorems, the algebraic assumptions on $\Phi$), we show that the class of linear maps $L$ from {\bf Step 1} actually equals the entire conformal orthogonal group $CO(d)$. The deduction of this from the assumption of non-linearity is recorded in Proposition \ref{prop-keyprop} which is perhaps the main technical novelty of this paper. 

{\bf Step 3.} This step is an adaptation of \cite[Proposition 1.38]{Hochmanpreprint} which in turn uses the local entropy averages machinery developed in Hochman and Shmerkin's seminal work \cite{HochmanShmerkin2012}. Using the invariance established in {\bf Steps 2} and {\bf 3} together with a projection theorem of López and Moreira \cite{LopezMoreira2015} for product sets, we show that the dimension of (the ergodic components of) tangent distributions of $\mu\times\nu$ is preserved under the linear map $\R^d\times\R^d \to\R^d$, $(x,y)\mapsto x+y$. This is done in Proposition \ref{prop-FDprojections}. Consequently, the dimension of $\mu\times\nu$ is preserved under the same map. Here the tangent-regularity of $\nu$ is crucially used, exactly as Ahlfors regularity was used in \cite[Theorem 1.39]{Hochmanpreprint}, to deduce that the dimension of ergodic components of tangent distributions of $\mu\times\nu$ equal the dimension of $\mu\times\nu$. 

The proof of Theorem \ref{thm-8} also follows steps 1 and 2 with the obvious change that instead of the product measure $\mu\times\nu$, we only consider the measure $\mu$ and its tangent distributions. Afterwards, using the invariance $P = P\circ L^{-1}$ established in step 2, an application of Marstrand-Mattila projection theorem for measures shows that the dimension of $P$ is preserved under orthogonal projections.

\section{Preliminaries}

Let $U\subset \R^d$ be an open and convex set. Denote by $\mathcal C_{\rm conf}(U)$ the set of injective, continuously differentiable functions $f: U\to \R^d$ such that $\overline{f(U)}\subset U$, $\frac{Df(x)}{\Vert Df(x)\Vert} \in O(d)$ for every $x\in U$, and the function $x\mapsto Df(x)$ is $\alpha$-Hölder for some $\alpha>0$. Elements of $\mathcal C_{\rm conf}(U)$ are called \emph{conformal maps} of $U$. For $d=2$, conformal maps are complex analytic, and for $d\geq 3$, they are Möbius transformations by Liouville's theorem. A function $f\in \mathcal C_{\rm conf}(U)$ is contractive if $\sup_{x\in U} \Vert Df_i(x)\Vert < 1$. For a contractive conformal map $f\in \mathcal C_{\rm conf}(U)$ with unique fixed point in $x_0 \in U$ we introduce the following notations:
\begin{alignat*}{2}
    \lambda_f(x) &:= \Vert Df(x)\Vert, \qquad &x\in U\\
    O_f(x) &:= Df(x) \Vert Df(x)\Vert^{-1}, \qquad &x\in U\\
    \lambda(f) &:= \lambda_f(x_0), \\
    O(f) &= O_f(x_0).
\end{alignat*}
Throughout the paper, $\R^d$ is equipped with the maximum metric and $B(x,t)$ denotes the closed ball centered at $x$ and of radius $t$. In particular, $B(0,1) = [-1,1]^d$. 

We begin with a tool that allows us to deduce dimensional results for self-conformal sets from dimensional results for self-conformal measures. 

\begin{lemma}\label{lemma-supportsmeasure}
    Let $\Phi$ be a conformal iterated function system on $\R^d$ with attractor $X$. For any $\varepsilon>0$, there exists $k\in\N$ and a conformal iterated function system $\Phi' \subset \lbrace f_{i_1}\circ \cdots \circ f_{i_k}:\ i_1, \ldots, i_k\in\Gamma\rbrace$ and a self-conformal measure $\mu$ associated to $\Phi'$ such that
    \begin{equation*}
        \dim \mu \geq \dim X - \varepsilon.
    \end{equation*}
    Moreover, $\Phi'$ can be chosen so that
    \begin{enumerate}
    \item $\Phi'$ is totally non-linear, if $\Phi$ is totally non-linear,
    \item the attractor of $\Phi'$ is not contained in a $C^1$-hypersurface, if $X$ is not contained in a $C^1$-hypersurface.
    \end{enumerate}
\end{lemma}

\begin{proof}
    Let $\varepsilon>0$. By \cite[Proposition 2.5]{FraserPollicott2015}, for any large enough $k\in\N$ there exists a conformal iterated function system $\Phi_k \subset \lbrace f_{i_1}\circ \cdots \circ f_{i_k}:\ i_1, \ldots, i_k\in\Gamma\rbrace$ with strong separation property and attractor $X_k$ such that $\dim X_k \geq \dim X - \varepsilon/3$. The process of constructing a self-conformal measure on $X_k$ of dimension at least $\dim X_k - \varepsilon/3$ is standard and carried out for example in \cite[Proof of Corollary 1.3]{Rapaport2023}.

    Let us then prove the property (1). We begin with the following
    \begin{claim}
        Let $h:U\to h(U)\subset U$ be an invertible, differentiable contraction. If $h^n$ is affine for some $n\in\N$, then also $h$ is affine.
    \end{claim}
    \begin{proof}[Proof of Claim]
        Suppose that for some $n\in\N$, $h^n(x) = f(x)$ where $f:U\to \R^d, f(x)\mapsto Ax +b$ for some invertible linear map $A$ with $\Vert A\Vert < 1$ and $b\in\R$. Then 
        \begin{equation*}
            h(Ax + b) = h^{n+1}(x) = A h(x) + b
        \end{equation*}
        for every $x\in U$. Differentiating this we obtain $Dh(Ax + b)A = ADh(x)$, or, using invertibility of $A$,
        \begin{equation*}
            Dh(f(x)) = Dh(x)
        \end{equation*}
        for every $x\in U$. Iterating this, $Dh(f^n(x)) = Dh(f^{n-1}(x))=\ldots = Dh(x)$. Since $f^n(x)$ converges to the fixed point of $f$ for every $x\in U$, it follows that $Dh$ is constant.
    \end{proof}
    It follows from the above claim that the system $\Phi^n :=\lbrace f_i^n:\ i\in\Gamma\rbrace$ is totally non-linear for every $n\in\N$. Let $n\in\N$ be large enough so that for any $f\in \Phi^{n}$, $\diam f(B(0,1)) \leq \min_{g,h\in\Phi_k} {\rm dist}(g([-1,1]^d), h([-1,1]^d))$. Then define $\Phi' = \Phi^{n} \cup \Phi_k \setminus\lbrace g\in \Phi_k:\ g([-1,1]^d)\cap \bigcup_{f \in \Phi^{n}} f([-1,1]^d) \neq \emptyset\rbrace$ and note that because of the strong separation condition and the choice of $n$, $\#\Phi' \geq \#\Phi_k - O_d(\#\Gamma)$, where $O_d$ depends only on the dimension of the ambient space. In particular, it is not difficult to see that when $k$ is large enough, the attractor $X'$ of $\Phi'$ satisfies $\dimh X' \geq \dimh X_k-\varepsilon/3$. This can be deduced for example from \cite[Propositions 9.6 and 9.7]{Falconerbook}. Since $\Phi'$ contains the totally non-linear system $\Phi^{n}$, also $\Phi'$ is totally non-linear. 

    It remains to prove (2). For any $k\in\N$, let $\Phi_k' = \{f_i\}_{i\in \Gamma_k}$ denote the totally non-linear IFS defined in the preceding paragraph with attractor $X_k'$. The systems $\Phi_k'$ can be chosen so that ${\rm dist}(X_k', X) \leq 1/k$. Suppose for a contradiction that for each $k$, the set $X_k'$ is contained in an analytic hypersurface of $\R^d$, say $H_k$. For $\mathtt i = (i_1,\ldots, i_n)\in\Gamma_k^* := \bigcup_{m\in\N} \Gamma^m_k$, let $T_\iii:\R^d\to\R^d$ denote the map $T_\iii(x) = \diam f_{\mathtt i}([-1,1]^d)^{-1} x - f_\iii(0)$, where $f_\mathtt i := f_{i_1} \circ \cdots \circ f_{i_n}$. Since $f(X_k')\subset H_k$ for each $f\in \Phi_k'$, we have $T_\iii(f_\iii(X_k')) \subset T_\iii(H_k)\cap [-1,1 ]^d$ for every $\iii\in\Gamma_k^*$. Since $H_k$ is a $C^1$-hypersurface, when $\mathtt i\in\Gamma^*$ is long enough the set $T_{\iii}(H_k)\cap [-1,1]^d$ is contained in a $1/k$-neighbourhood of a hyperplane. That is, there exist sequences of words $(\iii_k)_{k\in\N}\in\Gamma^*$ and hyperplanes $(V_k)_{k\in\N}\subset \R^d$ such that $T_{\iii_k}( f_{\iii_k}(X_k')) \subset V_k^{1/k}$ for every $k$, where $(\cdot)^{1/k}$ denotes the $1/k$-neighbourhood. By the bounded distortion property (see e.g. \cite[Lemma 2.1]{Patzschke1997}), the family $\lbrace T_{\iii} \circ f_{\iii}:\ \iii\in\Gamma^*\rbrace$ is equicontinuous, in particular, by Arzelá-Ascoli theorem the sequence $T_{\iii_k} \circ f_{\iii_k}$ has an analytic accumulation point $g$ in the topology of uniform convergence. In particular, for any $\varepsilon>0$ and large enough $k$, 
    \begin{equation*}
        g(X) \subset g(X_k')^\varepsilon\subset (T_{\iii_k}(f_{\mathtt i_k}(X_k')))^{2\varepsilon} \subset (V_k)^{2\varepsilon + 1/k}
    \end{equation*}
    which means that $g(X)$ is contained in a hyperplane. Since $X$ is then contained in the image of the hyperplane under $g^{-1}$, this is a contradiction.
\end{proof}

\subsection{Scaling sceneries} Let $\mathcal P(A)$ denote the set of Borel probability measures on a metrizable topoligical space $A$. For $\mu\in \mathcal P(B(0,1))$, $x\in {\rm spt}(\mu)$ and $t\geq 0$, let $\mu_{x,t}$ denote the probability measure on $B(0,1)$ given by 
\begin{equation*}
    \mu_{x,t}(A) = \frac{\mu(2^{-t}A + x)}{\mu(B(x,2^{-t}))}.
\end{equation*}
Weak-$^*$ accumulation points of $(\mu_{x,t})_{t\geq 0}$ are called tangent measures of $\mu$ at $x$. For $x\in {\rm spt}(\mu)$ and $T\geq 1$, define the \emph{scenery of $\mu$ up to time $T$} as
\begin{equation*}
    \langle \mu\rangle_{x,T} = \frac{1}{T}\int_0^T \delta_{\mu_{x,t}}\,dt \in\mathcal P(\mathcal P(B(0,1))).
\end{equation*}
The weak-$^*$ accumulation points of $(\langle \mu\rangle_{x,T})_{T\geq 1}$ in $\mathcal P(\mathcal P(B(0,1)))$ are called \emph{tangent distributions} of $\mu$ at $x$. If there is unique accumulation point $P$, we say that $\mu$ generates $P$ at $x$. For any $P\in \mathcal P(\mathcal P(B(0,1)))$, we define the \emph{dimension} of $P$ as 
\begin{equation*}
    \dim P := \int \dim \nu\,dP(\nu).
\end{equation*}
A Borel probability measure $P \in \mathcal P(\mathcal P(B(0,1)))$ is a \emph{fractal distribution} if $S_t P = P$ for every $t\geq 0$, and if for any $A\subset \mathcal P(B(0,1))$ with $P(A) = 1$, for $P$-almost every $\mu\in A$ and $\mu$-almost every $x\in B(0,1)$ it holds that $\mu_{x,t}\in A$ for $t\geq 0$ such that $B(x,2^{-t})\subset B(0,1)$. A fractal distribution $P$ is called \emph{ergodic} if it is ergodic under the semi-flow $(S_t)_{t\geq 0}$ on $\mathcal P(\mathcal P(B(0,1)))$.

For $t\geq 0$, let $S_t$ denote the map $\mu\mapsto \mu_{0,t}$ defined on measures $\mu$ with $0\in {\rm spt}(\mu)$. Tangent distributions are always invariant under $S_t$ for every $t\geq 0$ \cite{Hochmanpreprint}. Slightly abusing notation, for any $f:\R^d\to\R^d$ we let $fP$ denote the push-forward of $P$ under $\mu\mapsto f\mu = \mu\circ f^{-1}$. 

The following result can be extracted from \cite[Proof of Proposition 1.9]{Hochmanpreprint}.

\begin{proposition}\label{prop-imagedistribution}
    Let $\mu$ be a Radon measure on $\R^d$, $x\in\R^d$, and suppose that $\lim_{i\to\infty} \langle\mu\rangle_{x, T_i} = P$ in the weak-$^*$ topology for some tangent distribution $P$. Then, for any diffeomorphism $f: \R^d\to\R^d$ and all sufficiently large $t_0\geq 0$,
    \begin{equation*}
        \lim_{i\to\infty}\langle f\mu\rangle_{f(x), T_i} = S_{t_0} Df(x) P.
    \end{equation*}
\end{proposition}

The involvement of $S_{t_0}$ here is necessary to ensure that $Df(x)P$ is supported on probability measures which are supported on $[-1,1]^d$. The following result can be extracted from \cite[Proof of Proposition 3.8]{Hochmanpreprint}. We use $d_{\rm LP}$ to denote the Lévy-Prokhorov metric in the space of probability measures.

\begin{proposition}\label{prop-densitytheorem}
    Let $\mu$ and $\nu$ be non-trivial Radon measures on $\R^d$ such that $\mu\ll\nu$. Then, for $\mu$-almost every $x\in\R^d$, 
    \begin{equation*}
        \lim_{t\to\infty} d_{\rm LP}(\mu_{x,t},\nu_{x,t}) = 0.
    \end{equation*}
\end{proposition}

For a Borel probability measure $\mu$ on $\R^d$ and $x\in {\rm spt}(\mu)$, let $TD(\mu,x)\subset \mathcal P(\mathcal P(B(0,1)))$ denote the set of tangent distributions of $\mu$ at $x$. It turns out that tangent distributions of a measure capture its local dimensions at almost every point. For every $x\in \R^d$, define
\begin{align*}
    \overline{\dim}_{\rm loc}\mu(x) := \limsup_{r\to 0}\frac{\log\mu(B(x,r))}{\log r},\quad \underline{\dim}_{\rm loc}\mu(x) := \liminf_{r\to0} \frac{\log\mu(B(x,r))}{\log r}.
\end{align*}
When these quantities agree almost everywhere and their common value is a constant, $\mu$ is called exact-dimensional.
\begin{proposition}\label{prop-localdimensions}
Let $\mu$ be a Borel probability measure on $\R^d$. Then for $\mu$-almost every $x$, 
    \begin{align*}
        \overline{\dim}_{\rm loc}\mu(x) &= \sup\lbrace \dim P:\ P \in TD(\mu,x)\rbrace, \\
        \underline{\dim}_{\rm loc}\mu(x) &= \inf \lbrace \dim P:\ P\in TD(\mu,x)\rbrace.
    \end{align*}
\end{proposition}

See \cite[Proposition 1.19]{Hochmanpreprint} and \cite[Remark after Theorem 3.21]{KaenmakiSahlstenShmerkin2015GMT}. The following is a consequence of the main result of \cite{HochmanShmerkin2012}, and the formulation below follows from \cite[Theorem 1.23]{Hochmanpreprint}.

\begin{theorem}\label{thm-hochmanprojections}
    Let $\mu$ be a Borel probability measure on $\R^d$. Then for every linear map $\pi$, 
    \begin{equation*}
        \dim \pi\mu \geq \essinf_{x\sim \mu} \inf_{P \in TD(\mu,x)} \dim \pi P = \essinf_{x\sim \mu} \inf_{P \in TD(\mu,x)} \int \dim \pi \nu\,dP(\nu).
    \end{equation*}
\end{theorem}

In the following theorem we collect some additional properties of fractal distributions.

\begin{theorem}\label{thm-hochmanrigidity}
    Let $\mu$ be a Borel probability measure on $\R^d$, and let $P\in\mathcal P(\mathcal P(B(0,1)))$ be a fractal distribution. Let $P = \int Q\,dP'(Q)$ denote the ergodic decomposition of $P$ with respect to $(S_t)_{t\geq 0}$. 
    \begin{enumerate}
        \item For $\mu$-almost every $x\in \R^d$, every tangent distribution of $\mu$ at $x$ is a fractal distribution.
        \item $P'$-almost every $Q$ is an ergodic fractal distribution.
        \item $P$-almost every measure $\nu$ is exact-dimensional. Moreover, there exists an ergodic component $Q$ such that for $\nu$-almost every $x\in \R^d$, $\langle \nu\rangle_{x,T} \to Q$ as $T\to \infty$. 
        \item For $P'$-almost every $Q$ and any linear map $\pi$, the functions $\nu\mapsto \dim \nu$ and $\nu\mapsto \dim\pi\nu$ are constant $Q$-almost everywhere. 
    \end{enumerate}
\end{theorem}
\begin{proof}
    Claim (1) is \cite[Theorem 1.7]{Hochmanpreprint}, claim (2) is \cite[Theorem 1.3]{Hochmanpreprint}, claim (3) is \cite[Theorem 1.6]{Hochmanpreprint} and claim (4) is \cite[Lemma 1.18]{Hochmanpreprint}.
\end{proof}

If $\nu\in\mathcal P(\mathcal C_{\rm conf}((-1,1)^d))$ and $\mu\in\mathcal P((-1,1)^d)$, we define their \emph{generalized convolution} $\nu\cdot\mu$ by 
\begin{equation*}
    \nu\cdot\mu(A) = \int h\mu(A) \,d\nu(h).
\end{equation*}

The following description for tangent measures of self-conformal measures was provided in \cite{BaranyKaenmakiPyoralaWu2023}.

\begin{lemma}[Lemma 3.5 of \cite{BaranyKaenmakiPyoralaWu2023}]\label{lemma-tangentstructure}
    Let $\mu$ be a self-conformal measure on $\R^d$. For $\mu$-almost every $x$ and every tangent distribution $P$ at $x$, for $P$-almost every $\eta$, there exists a measure $\gamma\in\mathcal P(\mathcal C_{\rm conf}((-1,1)^d))$ such that $\eta = (\gamma\cdot\mu)_{B(0,1)}$. 
\end{lemma}

The following proposition is a corollary of \cite[Proposition 3.4]{BaranyKaenmakiPyoralaWu2023} and allows us to relate the sceneries of tangents of $\mu$ back to the sceneries of $\mu$. 

\begin{proposition}\label{prop-mainprop}
    Let $\nu\in\mathcal P(\mathcal C_{\rm conf}((-1,1)^d))$ and $\mu\in\mathcal P((-1,1)^d)$ be such that the measures $\mu$ and $\nu\cdot \mu$ are exact-dimensional, and $\dim \nu\cdot\mu = \dim \mu>0$. Then 
    \begin{equation*}
        \lim_{n\to\infty} \iint \frac{1}{n}\int_1^n d_{\rm LP}((\nu\cdot\mu)_{y,t}, (h_*\mu)_{y,t})\,dt\,dh\mu(y) \,d\nu(h) =0.
    \end{equation*}
\end{proposition}
\begin{proof}
    Although the deduction is contained in \cite[Proof of Theorem 1.1]{BaranyKaenmakiPyoralaWu2023}, we include the short explanation on how the claim follows from \cite[Proposition 3.4]{BaranyKaenmakiPyoralaWu2023}. Since $\nu$ is supported on conformal maps and $\mu$ is exact-dimensional, it follows that
    \begin{equation*}
        \liminf_{n\to\infty} \frac{1}{n} \sum_{k=1}^n H((h\mu)^{\mathcal D_k(x)}, \mathcal{D}_1) = \dim \mu
    \end{equation*}
    for $h\mu$-almost every $x$, see \cite[Lemma 6.7]{Hochmanpreprint}. Therefore, by Egorov's theorem, for any $0<\varepsilon<\delta$ there exists a set $A\subset \mathcal{C}_{\rm conf}((-1,1)^d)$ with $\nu(A)\geq 1-\delta$ and an integer $N$ such that 
    \begin{equation*}
        h\mu(\lbrace x \in \R^d:\ \frac{1}{n} \sum_{k=1}^n H((h\mu)^{\mathcal D_k(x)}, \mathcal{D}_1) \geq \dim\mu - \varepsilon\rbrace) \geq 1-\delta\quad \text{for}\ \nu_A\text{-a.e.}\ h
    \end{equation*}
    for every $n\geq N$; This is exactly assumption (1) of \cite[Proposition 3.4]{BaranyKaenmakiPyoralaWu2023} for $\nu_A$ in place of $\nu$ and $\dim\mu-\varepsilon$ in place of $\alpha$. On the other hand, since $\nu\cdot\mu$ and therefore also $\nu_A\cdot \mu$ is exact-dimensional with $\dim\nu_A\cdot \mu = \dim\nu\cdot \mu$, as long as $N$ is large enough we have
    \begin{equation*}
        \frac{1}{n} H_n(\nu_A\cdot\mu, \mathcal D_n) \leq  \dim \nu\cdot\mu - \varepsilon+\delta = (\dim \mu - \varepsilon)+\delta
    \end{equation*}
    for every $n\geq N$. This is assumption (2) of \cite[Proposition 3.4]{BaranyKaenmakiPyoralaWu2023}. Therefore the said proposition says that
    \begin{equation*}
    \iint \frac{1}{n}\int_1^n d_{\rm LP}((\nu_A\cdot\mu)_{y,t}, (h_*\mu)_{y,t})\,dt\,dh\mu(y) \,d\nu(h) <\varepsilon
    \end{equation*}
    for every $n\geq N$. Since $\nu(A)\geq 1-\delta$ and $\delta$ can be taken arbitrarily small, the claim follows by applying Proposition \ref{prop-densitytheorem} for $\nu_A\cdot \mu\ll\nu\cdot\mu$.
\end{proof}

The fact that self-conformal measures are tangent-regular follows from Lemma \ref{lemma-tangentstructure}.

\begin{lemma}\label{lemma-tangentregular}
    Self-conformal measures are tangent-regular.
\end{lemma}
\begin{proof}
    Let $\Phi$ be conformal iterated function system on $\R^d$, and let $\mu$ be a self-conformal measure associated to $\Phi$. By Lemma \ref{lemma-tangentstructure}, for $\mu$-almost every $x$ and any tangent distribution $P$ at $x$, $P$-almost every measure $\eta$ has the form 
    \begin{equation*}
        \eta = (\gamma\cdot\mu)_{B(0,1)}
    \end{equation*}
    for some $\gamma \in\mathcal P(\mathcal C_{\rm conf}((-1,1)^d))$. Since $\dim (\gamma\cdot\mu)_{B(0,1)}\geq \dim\mu$, and Proposition \ref{prop-localdimensions} asserts that strict inequality cannot hold for $P$-almost every choice of $\eta$ since self-conformal measures are exact-dimensional \cite{FengHu2009}, it follows that $\dim \eta = \dim\mu$ for $P$-almost every $\eta$. 
\end{proof}

\subsection{Hypersurfaces}

Below we collect a few rigidity results on subsets of hypersurfaces. The first is a simple consequence of the implicit function theorem.

\begin{lemma}\label{lemma-hypersurface}
    Let $A\subset \R^d$ be a closed set. If $A\times A$ is contained in a countable union of smooth hypersurfaces of $\R^d\times\R^d$, then there exists an open set $U\subset \R^d$ such that $A\cap U$ is contained in a smooth hypersurface of $\R^d$.
\end{lemma}
\begin{proof}
    By Baire category theorem, there exists a smooth hypersurface $M\subset \R^d\times\R^d$ and an open set $V_0\subset \R^d\times\R^d$ such that $(A\times A)\cap V_0\subset M$. By \cite[Proposition 5.16]{Lee2013}, there exists an open set $V_1 \subset V_0$ and a smooth function $F: V_1\to \R$ such that $Df(x)\neq 0$ for every $x\in V_1$ and $F^{-1}(0) = V_1 \cap M$. Since not all partial derivatives of $f$ are identically zero, by reordering the coordinates (in such a way that $A\times A$ remains a product set) we may suppose that for some $x_0 \in V_1\cap M$ the partial derivative with respect to the last coordinate is non-zero at $x_0$. 
    
    Applying the implicit function theorem for $F$ at $x_0$, there exists an open set $V_2\subset \R^{d}\times \R^{d-1}$ a smooth map $g:V_2\to \R$ such that 
    \begin{equation*}
        F^{-1}(0) \cap (V_2\times \R) = \lbrace (z,g(z)):\ z \in V_2\rbrace.
    \end{equation*}
    Let $V_3\subset \R^d$, $V_4\subset \R^{d-1}$ and $V_5\subset \R$ be open sets such that $V_3\times V_4\subset V_2$ and $V_3 \times V_4 \times V_5 \subset V_1$. Then 
    \begin{align*}
        (A\times A) \cap (V_3\times V_4\times V_5) \cap M &\subset F^{-1}(0)\cap (V_3\times V_4\times V_5) \\
        &\subset \lbrace (x,y,g(x,y)):\ x\in V_3, y\in V_4\rbrace.
    \end{align*}
    Fixing an arbitrary $a\in A\cap V_3\subset \R^d$, we find further that 
    \begin{equation*}
        \lbrace a\rbrace\times (A \cap (V_4\times V_5)) \subset \lbrace (a,y,g(a,y)):\ y\in V_4\rbrace.
    \end{equation*}
    In particular, the set $A\cap (V_4\times V_5) \subset \R^d$ is contained in the graph of $\tilde{g}: V_4 \to \R$, $\tilde{g}(y) = g(a,y)$ which is a smooth hypersurface of $\R^d$. This completes the proof with $U := V_4\times V_5$. 
\end{proof}

When the set $A$ is additionally self-conformal, it is easy to upgrade the statement to a global one.

\begin{lemma}\label{lemma-hypersurface_selfconformal}
    Let $K\subset \R^d$ be a self-conformal set. If $K\times K$ is contained in a countable union of smooth hypersurfaces of $\R^d\times\R^d$, then $K$ is contained in a smooth hypersurface of $\R^d$.
\end{lemma}

\begin{proof}
    Let $U\subset \R^d$ be the open set given by Lemma \ref{lemma-hypersurface}, such that $K\cap U$ is contained in a smooth hypersurface $M\subset \R^d$. Iterating \eqref{eq-selfconformalset}, we find a conformal map $f: (-1,1)^d\to \R^d$ such that $f(K) \subset K\cap U\subset M$, so $K\subset f^{-1}(M)$ which is a smooth hypersurface of $\R^d$.
\end{proof}

Finally, we require the following rigidity result of Käenmäki for self-conformal sets. While the result is not formulated in this exact form in \cite{Kaenmaki2003}, it can be extracted from the proof of \cite[Theorem 2.1]{Kaenmaki2003}. We provide the short proof for completeness.

\begin{lemma}\label{lemma-selfconformalrigidity}
    Let $K\subset \R^d$ be a self-conformal set. Suppose that there exists a $C^1$-hypersurface $M\subset \R^d$ such that $K\subset M$. Then,
    \begin{enumerate}
        \item if $d=2$, the set $M$ is an analytic curve, or
        \item if $d\geq 3$, the set $M$ is either an affine hyperplane or a sphere.
    \end{enumerate}
\end{lemma}
\begin{proof}
    Let $\Phi = \lbrace f_i\rbrace_{i\in\Gamma}$ be the conformal iterated function system defining $K$. For $\mathtt i = (i_1,\ldots, i_n)\in\Gamma^* := \bigcup_{m\in\N} \Gamma^m$, let $T_\iii:\R^d\to\R^d$ denote the map $T_\iii(x) = \diam f_{\mathtt i}([-1,1]^d)^{-1} x - f_\iii(0)$, where $f_\mathtt i := f_{i_1} \circ \cdots \circ f_{i_n}$. Since $f(K)\subset M$ for each $f\in \Phi$, we have $T_\iii(f_\iii(K)) \subset T_\iii(M)\cap [-1,1 ]^d$ for every $\iii\in\Gamma^*$. Since $M$ is a $C^1$-hypersurface, when $\mathtt i\in\Gamma^*$ is long enough the set $T_{\iii}(M)\cap [-1,1]^d$ is contained in a $1/k$-neighbourhood of a hyperplane. In particular, there exists a sequence $(\iii_k)_{k\in\N} \subset \Gamma^*$ and a hyperplane $V\subset \R^d$ such that $T_{\iii_k}( f_{\iii_k}(K)) \subset V^{1/k}$, where $(\cdot)^{1/k}$ denotes the $1/k$-neighbourhood. By the bounded distortion property (see e.g. \cite[Lemma 2.1]{Patzschke1997}), the family $\lbrace T_{\iii} \circ f_{\iii}:\ \iii\in\Gamma^*\rbrace$ is equicontinuous, in particular, by Arzelá-Ascoli theorem the sequence $T_{\iii_k} \circ f_{\iii_k}$ has an analytic accumulation point $g$ in the topology of uniform convergence. For this $g$ we have $K \subset g^{-1}(V)$. When $d=2$, analytic images of lines are analytic curves, which proves (1). When $d\geq 3$, $g$ is a Möbius transformation. Since the image of a hyperplane through a Möbius transformation is either a hypersphere or an affine hyperplane, this proves (2). 
\end{proof}

\subsection{The projection theorem}

The assumptions of non-linearity in Theorems \ref{thm-1} - \ref{thm-2} and the algebraic assumptions of Theorems \ref{thm-3} - \ref{thm-7} are utilised through a Marstrand-type projection theorem of López and Moreira for product sets. Below we record a special case of that result. Let $\mathcal L$ denote the Lebesgue measure on $\R$ and $\theta$ the unique bi-invariant probability measure on $SO(d)$. For $s\geq 0$ and Borel probability measure $\mu$ let 
\begin{equation*}
    I_s(\mu) = \iint |x-y|^{-s}\,d\mu(x)\,d\mu(y).
\end{equation*}

\begin{theorem}[Corollary of Theorem 2.3 of \cite{LopezMoreira2015}]\label{thm-moreira}
    Let $d, n\geq 2$. For every $\lambda = (t_1, O_1, \ldots, t_n, O_n) \in (\R\times SO(d))^n$, let 
    \begin{equation*}
    \pi_\lambda(x) = \sum_{i=1}^n 2^{t_i} O_i x_i.
    \end{equation*}
    Let $0\leq s_i \leq d$ for $1\leq i \leq n$ and $s = \min\lbrace \sum_{i=1}^n s_i, d\rbrace$. If $s \not\in \Z$, then there exist numbers $s_i'\leq s_i, 1\leq i \leq n$ and $C>0$ such that for any Borel probability measures $\mu_1,\ldots, \mu_n$ on $\R^d$ and for $(\mathcal L\times \theta)^n$-almost every $\lambda\in(\R\times SO(d))^n$, 
    \begin{equation*}
        I_s(\pi_\lambda(\mu_1\times\cdots\times\mu_n)) \leq C I_{s_1'}(\mu_1)\ldots I_{s_n'}(\mu_n).
    \end{equation*}
\end{theorem}
We will only apply this result with $n=2$. Let $\pi:\R^{d}\times\R^d\to \R^d$ denote the map $(x,y)\mapsto x+y$. Below we formulate an analogue of Theorem \ref{thm-moreira} for projections of fractal distributions with suitable invariance properties. For functions $f,g:\R^d\to R^d$, let $(f,g)$ denote the function $\R^d\times\R^d\to\R^d\times\R^d$, $(x,y)\mapsto (f(x),g(x))$. 
\begin{proposition}\label{prop-FDprojections}
    Let $P$ be an ergodic fractal distribution on $\R^{d}\times\R^d$ which is supported on product measures and satisfies the following: There exists $t_0\geq 0$ and a dense subset $\mathcal A\subset CO(d)$ such that $S_{t_0}(A,I)P = P$ for every $A \in \mathcal A$. Then
    \begin{equation*}
        \dim \pi P = \min\lbrace k, \dim P\rbrace.\qquad 
    \end{equation*}
\end{proposition}

\begin{proof}
It is well-known (see \cite[Section 8]{Mattila1995}) that if $I_s(\mu)<\infty$, then $\dim \mu \geq s$. Conversely, if $\dim \mu \geq s$, then for any $t<s$ and $\varepsilon>0$ there exists $E\subset \R^d$ with $\mu(E)>1-\varepsilon$ such that $I_t(\mu_E)<\infty$. For $P$-almost every $(\mu\times\nu)$, applying Theorem \ref{thm-moreira} to suitable restrictions of $\mu$ and $\nu$ with $s_1 < \dim \mu, s_2<\dim\nu$, we obtain 
\begin{equation*}
    \dim \pi(L_1\mu\times L_2\nu) = \min \lbrace d, \dim\mu+\dim\nu\rbrace
\end{equation*}
for $(\mathcal L\times \theta)^2$-almost every $(L_1,L_2) \in CO(d)^2 \cong (\R\times SO(d))^2$. Note also that
\begin{equation*}
\dim \pi(L_1\mu\times L_2\nu) = \dim L_2 \pi(L_2^{-1}L_1\mu \times \nu) = \dim\pi(L_2^{-1}L_1\mu\times\nu).
\end{equation*}
By \cite[Theorem 1.10]{HochmanShmerkin2012} (see also \cite[Theorem 1.21]{Hochmanpreprint}), the function $CO(d)\to \R$,
\begin{equation*}
    A\mapsto \int \dim \pi(A\mu \times\nu)\,dP(\mu\times\nu)
\end{equation*}
is lower semi-continuous. In particular, for any $L\in CO(d)$ and $(A_n)_n \subset \mathcal A$ with $\lim_n A_n = L$, 
\begin{align*}
    \int \dim \pi(L\mu \times\nu)\,dP(\mu\times\nu) &\leq \liminf_{n\to\infty} \int \dim \pi(A_n\mu\times\nu)\,dP(\mu\times\nu) \\
    &\leq \liminf_{n\to\infty} \int \dim \pi(S_t(A_n\mu\times\nu))\,dP(\mu\times\nu) \\
    &=\int \dim \pi(\mu\times\nu)\,dP(\mu\times\nu)
\end{align*}
by the $S_t(A,I)$-invariance of $P$. Since the projection of $(\mathcal L\times \theta)^2$ under $(L_1, L_2) \mapsto L_2^{-1}L_1$ is equal to $\mathcal L\times \theta$, we find that
    \begin{align*}
        \dim \pi P &= \int \dim\pi(\mu\times\nu) \,dP(\mu\times\nu) \\
        &\geq \iint \dim \pi(L\mu\times\nu)\,d(\mathcal L\times \theta)(L)\,dP(\mu\times\nu) \\
        &= \iint\min\lbrace d,\dim\mu + \dim\nu\rbrace\,dP(\mu\times\nu).
    \end{align*}
By Theorem \ref{thm-hochmanrigidity} (3) and (4), $\dim P = \dim(\mu\times\nu) = \dim \mu+\dim\nu$ for $P$-almost every $\mu\times\nu$. Since the inequality $\dim \pi P \leq \min\lbrace k,\dim P\rbrace$ is trivial, this completes the proof.
\end{proof}

\section{On the conformal orthogonal group}

Recall that $CO(d) =\lbrace r O:\ r>0, O \in SO(d)\rbrace$. In this section we establish conditions for certain subsets of $CO(d)$ to generate the entire group. If $G$ is a topological group and $\mathcal A\subset G$, we write 
\begin{equation*}
\langle \mathcal A\rangle := \overline{\lbrace A_1 A_2\cdots A_n:\ n\geq 1, A_i \in \mathcal A\rbrace}.
\end{equation*}
In this section, the group operations of both $CO(d)$ and $SO(d)$ are given by multiplication, and the group operation of $\R$ is given by addition. The group operation of $SO(d)\times \R$ is given by multiplication in the first coordinate and addition in the second. 

If $V$ is a vector space and $A\subset V$, then we write ${\rm span}(A)$ for the set of linear combinations of elements of $A$.

\begin{proposition}\label{prop-keyprop}
    Let $U \subset \R^d$ be open and convex, and let $E\subset U$ be such that for any open set $V\subset U$, the set $E\cap V$ is not contained in a smooth hypersurface of $\R^d$. Then for any $f \in \mathcal C_{\rm conf}(U)$ which is non-linear and orientation-preserving, we have
    \begin{equation*}
        \langle \lbrace Df(x)Df(y)^{-1}:\ x,y\in E \rbrace\rangle = CO(d).
    \end{equation*}
\end{proposition}

In the proof we utilise the following analogue of the identity principle for real-analytic functions of several variables. 

\begin{lemma}\label{lemma-mityagin}[Claim 2 in \cite{Mityagin2020}]
    Let $U\subset \R^d$ be open and convex. If $f: U\to \R$ is a real-analytic function which is not identically zero, then the set 
    \begin{equation}
        \lbrace x \in U:\ f(x) = 0\rbrace
    \end{equation}
    is contained in a countable union of smooth hypersurfaces of $\R^d$. 
\end{lemma}

\begin{proof}[Proof of Proposition \ref{prop-keyprop}]
    Let $U$, $E$ and $f$ be as in the statement. Since $CO(d) \cong SO(d) \times \R$ via $r O \mapsto (O, \log r)$, it suffices to show that 
    \begin{equation}\label{eq-goal}
        \langle \lbrace (O_f(x) O_f(y)^{-1}, \log\lambda_f(x) -\log\lambda_f(y)):\ x,y\in E\rbrace\rangle = SO(d)\times \R.
    \end{equation}     
    Recall that here $O_f(x) = Df(x)\Vert Df(x)\Vert^{-1} \in SO(d)$ and $\lambda_f(x) = \Vert Df(x)\Vert \in \R$. 
    
    Denote by $\mathfrak{so}(d)$ the Lie algebra of $SO(d)$. The set $\mathfrak{so}(d)$ consists of skew-symmetric $d\times d$-matrices and the exponential map $\mathfrak{exp}: \mathfrak{so}(d)\to SO(d)$ is given by the converging series $\mathfrak{exp}(A) = \sum_{k=1}^\infty \frac{A^k}{k!}$. By \cite[Proposition 20.8]{Lee2013}, $\mathfrak{exp}$ is a smooth diffeomorphism in a neighbourhood $X$ of $0 \in \mathfrak{so}(d)$. Denote $Y = \mathfrak{exp}(X)$ and let $\mathfrak{log}: Y \to X$ denote the inverse of $\mathfrak{exp}|_X$. Using continuity of $f$, by replacing $U$ with a subset we may assume that 
    \begin{equation*}
        O_f(x) O_f(y)^{-1} \in Y, \qquad x,y\in U.
    \end{equation*}
    In particular, 
    \begin{equation*}
        \mathfrak{log}(O_f(x) O_f(y)^{-1}) \in X, \qquad x,y\in U.
    \end{equation*}
    
    We also denote the exponential map $\mathfrak{so}(d) \times\R\to SO(d)\times \R,\ (v,x)\mapsto (\mathfrak{exp}(v), x)$ by $\mathfrak{exp}$. Since $G:=\langle \lbrace (O_f(x) O_f(y)^{-1}, \log\lambda_f(x) -\log\lambda_f(y)):\ x,y\in E\rbrace\rangle$ is a closed subgroup of $SO(d)\times \R$, it is a Lie subgroup of $SO(d)\times \R$ by the closed subgroup theorem, and its Lie algebra $\mathfrak{g}$ is a linear subspace of $\mathfrak{so}(d)\times \R$. Moreover, by \cite[Proposition 20.9]{Lee2013} the exponential map $\mathfrak{g} \to G$ is just the restriction of $\mathfrak{exp}$ to $\mathfrak{g}$. Therefore 
    \begin{equation*}
        (\mathfrak{log}(O_f(x) O_f(y)^{-1}), \log\lambda_f(x) - \log\lambda_f(y)) \in \mathfrak{g}, \qquad x,y\in U.
    \end{equation*}
    The crucial consequence of this inclusion is that also all linear combinations of the vectors $(\mathfrak{log}(O_f(x) O_f(y)^{-1}), \log\lambda_f(x) - \log\lambda_f(y))$ belong to $\mathfrak{g}$, in particular, the images of these linear combinations under $\mathfrak{exp}$ belong to $G$.
    
    Since $SO(d)$ is compact and connected, the exponential map $\mathfrak{exp}: \mathfrak{so}(d)\to SO(d)$ is surjective (see \cite[Corollary 11.10]{Hall2015}). Therefore, in order to prove \eqref{eq-goal} it suffices to show that 
    \begin{equation*}
        {\rm span}\lbrace (\mathfrak{log}(O_f(x) O_f(y)^{-1}), \log\lambda_f(x) -\log\lambda_f(y)):\ x,y\in E\rbrace = \mathfrak{so}(d) \times \R.
    \end{equation*}
    For this, on the other hand, it suffices to show that if $L: \mathfrak{so}(d)\times \R\to\R$ is any linear functional which is not identically zero, then the function
    \begin{equation}\label{eq-nonzerofunction}
        E\times E \to \R,\qquad (x,y)\mapsto  L(\mathfrak{log}(O_f(x) O_f(y)^{-1}),  \log\lambda_f(x) -\log\lambda_f(y))
    \end{equation}
    is not identically zero. Fix now an arbitrary non-trivial linear functional $L: \mathfrak{so}(d) \times \R \to \R$ and consider the function
    \begin{equation}\label{eq-nonzerofunction2}
    F: U\times U \to \R,\qquad (x,y) \mapsto L(\mathfrak{log}(O_f(x) O_f(y)^{-1}),  \log\lambda_f(x) -\log\lambda_f(y)).
    \end{equation}
    We note that $F$ is real-analytic as a composition of the real-analytic maps 
    \begin{alignat*}{2}
    &U\times U\to \R^d,\ &&(x,y) \mapsto \log\lambda_f(x) -\log\lambda_f(y), \\
    &U\times U \to Y,\ &&(x,y) \mapsto O_f(x) O_f(y)^{-1}, \\
    &Y\mapsto X,\ &&A \mapsto \mathfrak{log}(A), \\
    &X\times \R\to \R,\ &&(v_1, v_2) \mapsto L(v_1, v_2).
    \end{alignat*}
    We claim next that $F$ is not identically zero. Since $L$ is non-trivial, this will follow if we manage to show that
    \begin{equation}\label{eq-secondgoal}
         \Lambda := {\rm span}\lbrace (\mathfrak{log}(O_f(x) O_f(y)^{-1}), \log\lambda_f(x) -\log\lambda_f(y)):\ x,y\in U\rbrace = \mathfrak{so}(d)\times \R.
    \end{equation}
    Note that here $x$ and $y$ are chosen from the open set $U$ so this equality does not yet imply our goal \eqref{eq-goal}.
    
    Suppose first that $d=2$. Then $\mathfrak{so}(2) = \R$, $\mathfrak{log}(x) = x$ and $f$ is a complex differentiable function with $O_f(x) = {\rm Arg}(f'(x))$, $\lambda_f(x) = |f'(x)|$ and
    \begin{equation*}
        \Lambda = {\rm span}\lbrace ({\rm Arg}(a) {\rm Arg}(b)^{-1}, \log |a| - \log |b|):\ a,b\in f'(U)\rbrace.
    \end{equation*}
    Since $f'(U)$ is open, by varying $a$ and $b$ first along a line through origin and then along a circle centered at the origin we find that $\Lambda = \R^2 = \mathfrak{so}(2)\times\R$. 

    Suppose then that $d\geq 3$. By Liouville's theorem, $f$ is a non-linear (orientation-preserving) Möbius transformation; In particular, its derivative has the form
        \begin{equation*}
            Df(x) = \Vert x - a\Vert^{-2} A R_{x-a}, \qquad x\in U,
        \end{equation*}
    where $a\in \R^d$, $A \in SO(d)$ and $R_{x-a}$ is the reflection with respect to the hyperplane perpendicular to $x-a$. Varying $x\in U$ in a line which contains $a$, the matrix $O_f(x) O_f(y)^{-1} = A R_{x-a} R_{y-a}^{-1} A^{-1}$ remains constant and we find that $(0,r)\in\Lambda$ for some $r>0$. Thus, in order to complete the proof of \eqref{eq-secondgoal} it remains to show that  
    \begin{equation}\label{eq-thirdgoal}
        {\rm span}\lbrace \mathfrak{log}(A R_{x-a} R_{y-a}^{-1} A^{-1}):\ x,y\in U\rbrace = \mathfrak{so}(d).
    \end{equation}
    Note that $A R_{x-a} R_{y-a}^{-1}A^{-1}$ is rotation in the plane spanned by $A(x-a)$ and $A(y-a)$ by twice the angle between $x-a$ and $y-a$. Denote $O(x,y) := A R_{x-a} R_{y-a}^{-1}A^{-1}$. For all $x,y$ in a dense subset $U'\subset U$, $O(x,y)$ is an irrational rotation, in particular, $\langle O(x,y) \rangle$ is a one-dimensional Lie subgroup of $SO(d)$. Let 
    \begin{equation}\label{eq-lxy}
    \mathfrak{L}(x,y) := \lbrace A \in \mathfrak{so}(d):\ \mathfrak{exp}(t A) \in \langle O(x,y)\rangle\ \text{for every}\ t\geq 0 \rbrace
    \end{equation}
    denote its Lie algebra which is a one-dimensional linear subspace of $\mathfrak{so}(d)$; see \cite[Proposition 20.9]{Lee2013}. Again, the exponential map $\mathfrak{L}(x,y)\to SO(d)$ is just $\mathfrak{exp}|_{\mathfrak{L}(x,y)}$.
    
    Now, the second exterior power of $\R^d$, $\Wedge^2 \R^d = \lbrace v\wedge w: v,w\in\R^d\rbrace$, is isomorphic to $\mathfrak{so}(d)$, the space of skew-symmetric linear maps $\R^d\to \R^d$, with isomorphism given by the map 
    \begin{equation*}
    \Wedge^2\R^d \to \mathfrak{so}(d),\ v\wedge w \mapsto (x\mapsto \langle v,x\rangle w + \langle w,x \rangle v).
    \end{equation*}
    Since the skew-symmetric map 
    \begin{equation*}
    A(x-a)\wedge A(y-a):\ u\mapsto \langle A(x-a), u\rangle A(y-a) + \langle A(y-a),u \rangle A(x-a)
    \end{equation*}
    preserves the plane spanned by $A(x-a)$ and $A(y-a)$ and maps its orthogonal complement to $0$, it follows that $\mathfrak{exp}(A(x-a)\wedge A(y-a))$ is a rotation in the plane spanned by $A(x-a)$ and $A(y-a)$. In particular,
    \begin{equation*}
    \mathfrak{exp}(t\cdot (A(x-a)\wedge A(y-a))) \in \langle A R_{x-a} R_{y-a}^{-1}A^{-1} \rangle, \qquad t\geq0
    \end{equation*}
    and so $A(x-a)\wedge A(y-a) \in \mathfrak{L}(x,y)$ by definition; recall \eqref{eq-lxy}. Since $\mathfrak{L}(x,y)$ is a one-dimensional linear subspace of $\Lambda^2\R^d$ and $\mathfrak{log}(A R_{x-a} R_{y-a}^{-1}A^{-1}) \in \mathfrak{L}(x,y)$, we have 
    \begin{equation*}
    \mathfrak{log}(A R_{x-a} R_{y-a}^{-1}A^{-1}) = \alpha(x,y) \cdot A(x-a)\wedge A(y-a)
    \end{equation*}
   for some $\alpha(x,y)\in\R\setminus\lbrace0\rbrace$. Since $U'$ spans $\R^d$, the set $\lbrace A(x-a)\wedge A(y-a):\ x,y\in U'\rbrace$ spans $\Wedge^2\R^d \cong \mathfrak{so}(d)$ which completes the proof of \eqref{eq-thirdgoal}.

    We may now conclude the proof of the proposition. Since the map $F$ from \eqref{eq-nonzerofunction2} is a non-trivial real-analytic function, Lemma \ref{lemma-mityagin} asserts that its zero set belongs to a countable union of smooth hypersurfaces of $\R^d\times \R^d$. Since $E$ is not locally contained in a hypersurface of $\R^d$ by assumption, Lemma \ref{lemma-hypersurface} asserts that $E\times E$ cannot be contained in the zero set of $F$, completing the proof.
\end{proof}

\begin{proposition}\label{prop-generationRd}
    Let $d\geq 3$. Let $(O_1,\ldots, O_n) \in SO(d)^n$, $(r_1, \ldots, r_n) \in \R^n$, and let $\Gamma = \langle \lbrace r_i O_i, r_i^{-1} O_i^{-1}:\ 1\leq i \leq n\rbrace\rangle \leq CO(d)$. Suppose that the following conditions hold:
    \begin{enumerate}
        \item $\langle \lbrace O_i, O_i^{-1}:\ 1\leq i\leq n\rbrace\rangle = SO(d)$, 
        \item $\langle \lbrace \log r_i, - \log r_i:\ 1\leq i\leq n\rbrace\rangle = \R$, and 
        \item $r I\in\Gamma$ for some $r>0$.
    \end{enumerate}
    Then $\Gamma = CO(d)$. 
\end{proposition}

\begin{proof}
    We view $\Gamma$ as a closed subgroup of $SO(d)\times \R$ through the identification $rO \sim (O, \log r)$. Let $L = \Gamma\cap (\lbrace {\rm Id}\rbrace \times \R)$ and note that by assumption (3), $L \neq \lbrace 0\rbrace$. In order to prove the proposition it suffices to show that $L = \R$. Since $L$ is a subgroup of $\R$, the function
    \begin{equation*}
        \psi: \pi_1(\Gamma) \to \R/L, \qquad \psi(O) = t + L
    \end{equation*}
    where $t\in \R$ is such that $(O,t)\in \Gamma$, is a well-defined group homomorphism. We claim that $\psi$ is continuous at the identity $I\in \pi_1(\Gamma)$. Suppose otherwise, that there exists a sequence $(O_i, t_i)_{i\in\N}\in \Gamma$ such that $\lim_{i\to\infty} O_i = I$ but the sequence $(\psi(O_i))_{i\in\N} = (t_i)_{i\in\N}$ does not converge to $0\in \R/L$. By passing onto a subsequence and using compactness of $\R/L$ (which follows from the fact that $L\neq \lbrace0\rbrace$), we may assume $\lim_{i\to\infty} t_i +L =\lim_{i\to\infty} \psi(O_i) = t+L \neq 0$. However, since $\lim_{i\to\infty} (O_i, t_i) = (I, t) \in \Gamma$, this contradicts the definition of $L$. In particular, $\psi$ is continuous at $I$, and as a group homomorphism it is continuous at every point in $\pi_1(\Gamma)$. Since $\pi_1(\Gamma)$ is dense in $SO(d)$ by the assumption (1), we may extend $\psi$ to a group homomorphism $\tilde{\psi}: SO(d) \to \tilde{\psi}(SO(d)) \leq \R/L$. Denote by $\mathfrak{L}$ the Lie algebra of $\tilde{\psi}(SO(d))$ which is a closed subgroup of $\R/L$. Then $\tilde{\psi}$ gives rise to a Lie algebra homomorphism $d\tilde{\psi}: \mathfrak{so}(d) \to \mathfrak{L}$. Now, $\mathfrak{so}(d)/\ker (d\tilde{\psi})$ is an ideal of $\mathfrak{so}(d)$ by \cite[Corollary 1.55]{Knapp1996}, and it is abelian since $\mathfrak{so}(d)/\ker(d\tilde{\psi})$ is isomorphic to the abelian Lie algebra $d\tilde{\psi}(\mathfrak{so}(d))$ by the first isomorphism theorem for Lie algebras. By semisimplicity of $\mathfrak{so}(d)$, $\mathfrak{so}(d)/\ker(d\tilde{\psi})$ therefore trivial, that is $\ker(d\tilde{\psi}) = \mathfrak{so}(d)$ and so $d\tilde{\psi} \equiv 0$. By e.g. \cite[Theorem 20.19]{Lee2013} this implies $\tilde{\psi}\equiv 0$.

    Recalling the definition of $\psi$, we find that 
    \begin{equation*}
        \Gamma = \pi_1(\Gamma) \times L.
    \end{equation*}
    Since $\pi_2(\Gamma) = L$ is dense in $\R$ by assumption (2) and $\Gamma$ is closed, this completes the proof.
\end{proof}

The only place where the assumption (3) was used in the proof was to deduce that $L\neq \lbrace 0\rbrace$ which in turn was required to ensure that $\R/L$ is compact. Without this assumption the continuity of $\psi$ at $I$ might fail, for example, if there exists a sequence $(O_i, t_i)\in\Gamma$ such that $\lim_{i\to\infty} O_i = I$ but $\lim_{i\to\infty} t_i = \infty$. However, we do not know whether (3) is necessary for the conclusion of Proposition \ref{prop-generationRd} to hold.

\section{Proofs of the main results}

The self-conformality of $\mu$ will be utilized in the proofs of Theorems \ref{thm-1} - \ref{thm-7} through the following proposition. Recall that for $f,g:\R^d\to \R^d$, $(f,g)$ denotes the map $(x,y)\mapsto (f(x), g(x))$.  

\begin{proposition}\label{prop-tangentdistributioninvariant}
    Let $\Phi$ be a conformal iterated function system on $\R^d$, let $K$ denote its attractor and let $\mu$ be a self-conformal measure associated to $\Phi$. Let $\nu$ be a Borel probability measure on $\R^d$. For $\mu\times \nu$-almost every $(x,y)$, the following holds: If $Q$ is any tangent distribution of $\mu\times \nu$ at $(x,y)$ with ergodic decomposition $Q = \int P\,dQ'(P)$, then for $Q'$-almost every $P$ there exists a conformal map $h: U \to \R^d$ and $t_0\geq 0$ such that for any $f\in \Phi$ and for $\mu$-almost every $x,y\in h(K)$,
    \begin{equation}
        S_t(Df_h(x), I) P = S_t(Df_h(y)^{-1}, I) P= P
    \end{equation}
    for every $t\geq t_0$, where $f_h := h\circ f\circ h^{-1}$. 
\end{proposition}

\begin{proof}
    The proof very closely follows that of \cite[Theorem 1.3]{BaranyKaenmakiPyoralaWu2023}. Let $Q$ be a tangent distribution of $\mu\times\nu$ at $(x,y)$. Since we are using the maximum metric in $\R^d\times\R^d$, the measures $(\mu\times\nu)_{(x,y),t}$ are product measures and so $Q$ is supported on product measures. Moreover, if $\pi_1,\pi_2:\R^d\times\R^d\to \R^d$ denote the orthogonal projections to the first and last $d$ coordinates, respectively, then $\pi_1 Q$ and $\pi_2 Q$ are tangent distributions of $\mu$ and $\nu$ at $x$ and $y$, respectively. Let $Q = \int P\,dQ'(P)$ denote the ergodic decomposition of $Q$ with respect to $(S_t)_{t\geq 0}$. Then by Theorem \ref{thm-hochmanrigidity} (3), every ergodic component $P$ is generated by some tangent measure of $\mu\times \nu$ at almost every point; Fix an ergodic component $P$ and let $\eta = \mu'\times\nu'$ be such a tangent measure.
    
   Applying Propositions \ref{prop-localdimensions} and \ref{lemma-tangentstructure}, we may choose $\eta = \mu'\times \nu'$ in such a way that 
   \begin{enumerate} 
   \item $\mu'$ and $\nu'$ are exact-dimensional, 
   \item $\mu'$ has the form $\mu' = (\gamma\cdot\mu)_{B(0,1)}$ for some $\gamma\in\mathcal P(\mathcal C_{\rm conf}((-1,1))^d)$, and \item $\dim \mu' = \dim\mu$. 
   \end{enumerate}
   Property (1) is ensured by Theorem \ref{thm-hochmanrigidity} (3) since $\pi_1 Q$ and $\pi_2 Q$ are tangent distributions of $\mu$ and $\nu$, respectively. Property (2) is exactly Proposition \ref{lemma-tangentstructure}. Property (3) follows from tangent-regularity of $\mu$.  
   
   Having chosen such an $\eta$, we have
    \begin{equation*}
        \eta = (\gamma\cdot\mu)_{B(0,1)} \times \nu' = \int (h\mu\times\nu')_{B(0,1)}\,d\gamma(h) = (\gamma'\cdot(\mu\times\nu'))_{B(0,1)},
    \end{equation*}
    when $\gamma' = \int \delta_{(h,I)} \,d\gamma(h) \in \mathcal P(\mathcal C_{\rm conf}((-1,1)^{2d}))$. Using exact-dimensionality of $\mu$ \cite{FengHu2009} and properties (1) and (3) of $\eta$, we find that $\dim \mu\times\nu' = \dim\mu + \dim\nu' = \dim\mu'\times\nu' = \dim \eta$. Therefore, we are in position to apply Proposition \ref{prop-mainprop} to find that
    \begin{equation*}
        \lim_{n\to\infty} \iint \frac{1}{n}\int_1^n d_{\rm LP}((\gamma'\cdot(\mu\times\nu'))_{(x,y), t}, (h\mu\times\nu'))_{(x,y),t})\,dt\,d(h\mu\times\nu')(x,y)\,d\gamma(h) = 0.
    \end{equation*}
    Since $L^1$-convergence implies pointwise convergence along a subsequence, we find a subsequence $(n_k)_{k\in\N}$ such that 
    \begin{equation*}
        \lim_{k\to\infty} \frac{1}{n_k}\int_1^{n_k} d_{\rm LP}((\gamma'\cdot(\mu\times\nu'))_{(x,y), t}, (g(\mu\times\nu'))_{(x,y),t})\,dt = 0
    \end{equation*}
    for $\gamma$-almost every $h$ and $h\mu\times\nu'$-almost every $(x,y)$. Recalling that $(\gamma'\cdot (\mu\times\nu'))_{B(0,1)} = \mu'\times\nu' = \eta$ which generates the ergodic fractal distribution $P$ almost everywhere, we find that for $\gamma$-almost every $h$,
    \begin{equation*}
        \langle h\mu \times \nu'\rangle_{(x,y), n_k} \to P \qquad\text{for almost every}\ (x,y)\in B(0,1).
    \end{equation*}
    Fix any such $h$. It follows from Proposition \ref{prop-imagedistribution} that for any $f\in\Phi$, $h\mu\times\nu'$-almost every $(x,y)\in B(0,1)$ and for any large enough $t\geq 0$,
    \begin{align}
        \langle h\circ f \circ h^{-1}(h\mu) \times \nu'\rangle_{(h(f(h^{-1}(x))), y), n_k} &\to S_t (D(h\circ f \circ h^{-1})(x), I)P, \label{eq-f}\\
        \langle h\circ f^{-1} \circ h^{-1}(h\mu) \times \nu'\rangle_{(h(f^{-1}(h^{-1}(x))), y), n_k} &\to S_t (D(h\circ f^{-1} \circ h^{-1})(x), I)P.\label{eq-f-1}
    \end{align}
    Using the self-conformality of $\mu$ we find that 
    \begin{equation*}
        h\circ f \circ h^{-1}(h\mu) \ll h\mu \quad {\rm and} \quad h\mu \ll h \circ f^{-1} \circ h^{-1}((h\mu)_{h(f(K))})
    \end{equation*}
   for every $f\in \Phi$. In particular, testing \eqref{eq-f} with $x \in h(K)$ and using Proposition \ref{prop-densitytheorem}, we find that 
   \begin{align*}
       P &= \lim_{k\to\infty} \langle h\mu\times\nu\rangle_{(h(f(h^{-1}(x))),y), n_k} \\
       &= \lim_{k\to\infty} \langle h\circ f \circ h^{-1}(h\mu)\times\nu\rangle_{(h(f(h^{-1}(x))),y), n_k} \\
       &=S_t (D(h\circ f \circ h^{-1})(x), I)P \\
       &= S_t(Df_h(x), I) P
   \end{align*}
   for large enough $t\geq 0$, where $f_h := h\circ f\circ h^{-1}$. Testing \eqref{eq-f-1} with $x = h(f(h^{-1}(z))) \in h(f(K))$ for any $z\in h(K)$, we find that
   \begin{align*}
       P &= \lim_{k\to\infty} \langle h\mu\times\nu\rangle_{(h(z),y), n_k} \\
       &= \lim_{k\to\infty} \langle h \circ f^{-1} \circ h^{-1}(h\mu)\times\nu\rangle_{(h(f^{-1}(h^{-1}(x))),y), n_k} \\
       &= S_t (D(h\circ f^{-1} \circ h^{-1})(x), I)P \\
       &= S_t (D(h\circ f^{-1} \circ h^{-1})(h(f(h^{-1}(z)))), I)P \\
       &= S_t(D f_h(z)^{-1}, I)P
   \end{align*}
   for large enough $t\geq 0$. This completes the proof. 
\end{proof}

We are now ready to prove the main results. Recall that $\pi:\R^d\times\R^d\to\R^d$ denotes the map $(x,y)\mapsto x+y$. In the following we suppose that $\dim\mu>0$ and $\dim\nu>0$ since otherwise the statements are trivial.

\begin{proof}[Proof of Theorem \ref{thm-1}]
    Let $Q$ be a tangent distribution of $\mu\times\nu$ at $(x,y)$ with ergodic decomposition $Q = \int P\,dQ'(P)$. Then $\pi_1 Q$ and $\pi_2 Q$ are tangent distributions of $\mu$ and $\nu$ at $x$ and $y$, so by Theorem \ref{thm-hochmanrigidity} (3) and the assumption of tangent-regularity of both $\mu$ and $\nu$, for $\mu\times\nu$-almost every $(x,y)$ and for $Q'$-almost every $P$ it holds that $\dim(\mu'\times\nu') = \dim \mu+\dim\nu$ for $P$-almost every $\mu'\times\nu'$. 
    
    We now apply Proposition \ref{prop-tangentdistributioninvariant} for $Q'$-almost every ergodic component $P$. Fix one such $P$. Let $h \in \mathcal C_{\rm conf}((-1,1)^d)$ be the conformal map given by Proposition \ref{prop-tangentdistributioninvariant}, and using the total non-linearity of $\mu$, choose $f \in \Phi$ such that the function $f_h := h \circ f \circ h^{-1} \in \mathcal C_{\rm conf}((-1,1)^d)$ is not affine. 
    
    Suppose first that $d\geq 2$. By the assumption, $\mu$ is not supported on a planar analytic curve, a hyperplane or a sphere, so by Lemmas \ref{lemma-hypersurface}, \ref{lemma-hypersurface_selfconformal} and \ref{lemma-selfconformalrigidity}, the measure $\mu_U$ is not supported on a hypersurface of $\R^d$ for any open set $U\subset \R^d$. Applying Proposition \ref{prop-keyprop} with $f_h$ in place of $f$ (if this map is orientation-reversing, we replace it by its second iterate), we find that $P$ is invariant under $S_t(L, I)$ for every $L$ in a dense subset of $CO(d)$, for some $t\geq0$. Therefore Proposition \ref{prop-FDprojections} asserts that $\dim \pi P = \min\lbrace d, \dim P\rbrace$. Finally, by Theorem \ref{thm-hochmanprojections}, since $P$ was an arbitrary ergodic component of an arbitrary tangent distribution of $\mu\times\nu$ at $(x,y)$, we obtain
    \begin{equation*}
        \dim(\mu*\nu) \geq \min\lbrace d, \dim P\rbrace = \min\lbrace d,\dim\mu+\dim\nu\rbrace.
    \end{equation*}
    Since the other direction of the inequality holds for any exact-dimensional Borel measures $\mu$ and $\nu$, this completes the proof. 

    Let us then cover the case $d=1$. By possibly replacing $f_h$ by its second iterate, we have $f_h' >0$. Then Proposition \ref{prop-tangentdistributioninvariant} says that for some $t\geq 0$,  
    \begin{equation}\label{eq-d1invariance}
        S_t(S_{\log f'(z)}, I) P = S_t(S_{-\log f'(z)}, I) P = P,\qquad z\in h(K).
    \end{equation}
    Since $f_h$ is not affine and real-analytic, there exists $z\in h(K)$ with $|f_h''(z)| >0$, in particular, $f_h'$ is strictly monotone in a neighbourhood $U$ of $z$. Since $\dim(h(K)\cap U)>0$, there exist $z_1, z_2\in h(K)\cap U$ such that $\frac{\log f'_h(z_1)}{\log f_h'(z_2)} \not\in \Q$. Combining this with \eqref{eq-d1invariance}, we find that 
    \begin{equation*}
        S_t(S_r, I)P = P,\qquad r>0.
    \end{equation*}
    Thus Proposition \ref{prop-FDprojections} asserts that $\dim \pi P = \min\lbrace 1,\dim P\rbrace$ and the proof concludes as in the case $d\geq2$. 
\end{proof}

\begin{proof}[Proof of Theorem \ref{thm-2}]
    If either $\mu$ or $\nu$ is not supported on an analytic curve, the statement follows from Theorem \ref{thm-1}. Suppose then that $\mu$ and $\nu$ are both supported on analytic curves which are not parallel lines. For $\mu\times\nu$-almost every $(x,y)$, if $\mu'$ and $\nu'$ are any tangent measures of $\mu$ and $\nu$ at $x$ and $y$, respectively, then $\mu'$ and $\nu'$ are supported on non-parallel lines. This follows from the identity principle of complex-analytic functions and the assumption $\dim\mu\dim\nu>0$. For any such $\mu'$ and $\nu'$, $\mu' * \nu'$ is an image of $\mu'\times\nu'$ under an affine map of rank $2$, in particular $\dim (\mu'*\nu') =\dim\mu'\times\nu'$. Therefore, for $\mu\times\nu$-almost every $(x,y)$ and for any tangent distribution $P$ at $(x,y)$, it holds that
    \begin{align*}
        \dim \pi P &= \int \dim \pi (\mu'\times\nu') \,dP(\mu'\times\nu')\\
        &= \int \dim(\mu'\times\nu') \,dP(\mu'\times\nu') \\
        &= \int (\dim \mu' + \dim\nu') \,dP(\mu'\times\nu') \\ 
        &=\dim \mu + \dim \nu.
    \end{align*}
    Here we also used Theorem \ref{thm-hochmanrigidity} (3) and tangent-regularity of $\mu$ and $\nu$. Since the point $(x,y)$ and the distribution $P$ were arbitrary, the statement follows from Theorem \ref{thm-hochmanprojections}.
\end{proof}

The proofs of Theorems \ref{thm-3} - \ref{thm-7} are minor variations of the proof of Theorem \ref{thm-1}.

\begin{proof}[Proof of Theorem \ref{thm-3}]
    As in the proof of Theorem \ref{thm-1}, let $P$ be any ergodic component of a tangent distribution of $\mu\times\nu$, and let $\mu'\times\nu'$ be a tangent measure of $\mu\times\nu$ which generates $P$ and has dimension $\dim\mu'\times\nu' = \dim\mu+\dim\nu$. Apply Proposition \ref{prop-tangentdistributioninvariant} to $P$ and let $h\in \mathcal C_{\rm conf}((-1,1)^d)$ be the obtained conformal map. Let $x_f\in K$ denote the unique fixed point of $f\in \Phi$. Then $D(h\circ f\circ h^{-1})(h(x_f)) = Df(x_f)$. For each $f\in \Phi$, approximating $x_f$ with a sequence $h(K) \ni x_n\to h(x_f)$ such that the conclusion of Proposition \ref{prop-tangentdistributioninvariant} holds for $x_n \in h(K)$, i.e. $S_t(D f_h(x_n), I)P = P$ for large enough $t$, we find that $P$ is invariant under a dense subset of
    \begin{equation*}
    \langle \lbrace S_t(\lambda(f)\lambda(g)^{-1}O(f)O(g)^{-1}, I):\ f,g\in \Phi\rbrace\rangle  = \lbrace (S_t L, I):\ L\in CO(d)\rbrace
    \end{equation*}
    for large enough $t$. We may now conclude the proof similarly as to the proof of Theorem \ref{thm-1}: Proposition \ref{prop-FDprojections} (the presence of an additional scaling $S_t$ here is irrelevant) asserts that $\dim \pi P = \min\lbrace d,\dim P\rbrace$, and since $P$ was an arbitrary ergodic component of an arbitrary tangent distribution, it follows from Theorem \ref{thm-hochmanprojections} that $\dim(\mu*\nu) \geq \min\lbrace d,\dim P\rbrace = \min\lbrace d,\dim\mu+\dim\nu\rbrace$. 
\end{proof}

\begin{proof}[Proof of Corollary \ref{thm-4}]
    Under the assumptions, it follows from Proposition \ref{prop-generationRd} that 
    \begin{equation*}
    \langle\lbrace \lambda(f)O(f), \lambda(f)^{-1}O(f)^{-1}:\ f\in\Phi\rbrace\rangle = CO(d).
    \end{equation*} 
    Thus the statement follows from Theorem \ref{thm-3}
\end{proof}

\begin{proof}[Proof of Corollary \ref{thm-5}]
    We will show that under the assumptions, 
    \begin{equation*}
    \langle\lbrace \lambda(f)O(f), \lambda(f)^{-1}O(f)^{-1}:\ f\in\Phi\rbrace \rangle =CO(2).
    \end{equation*}
    The statement will then follow from Theorem \ref{thm-3}.

    Denote $\Gamma := \langle\lbrace \lambda(f)O(f), \lambda(f)^{-1}O(f)^{-1}:\ f\in\Phi\rbrace \rangle$. Let $f_1,f_2,f_3,f_4$ be as in the assumption (2). Since $\Gamma$ contains the elements $\lambda(f_1)O(f_1)\lambda(f_2)^{-1}O(f_2)^{-1} = \lambda(f_1)\lambda(f_2)^{-1} I$ and $\lambda(f_3)\lambda(f_4)^{-1} I$, it follows from the assumption (2) that $\Gamma$ in fact contains $r I$ for every $r>0$. Combining this with the assumption (1), we may conclude that $\Gamma = CO(2)$ which completes the proof.  
\end{proof}

\begin{proof}[Proof of Corollary \ref{thm-6}] 
    Let $P$ be an ergodic component of a tangent distribution of $\mu\times\nu$ at almost any $(x,y)$. Since both $\mu$ and $\nu$ are tangent-regular we have $\dim P = \dim\mu+\dim\nu$. Applying Proposition \ref{prop-tangentdistributioninvariant} twice for $P$, we obtain functions $h_1,h_2\in \mathcal C_{\rm conf}((-1,1)^d)$ such that if $f_h := h\circ f\circ h^{-1}$, then for $h_1 \mu \times h_2\nu$-almost every $(x,y)\in B(0,1)$ and every large enough $t\geq 0$,
    \begin{align*}
        &S_t (Df_{h_1}(x), I) P = S_t(Df_{h_1}(x)^{-1}, I) P = P,\\
        & S_t(I, Dg_{h_2}(y)) P = S_t (I, Dg_{h_2}(y)^{-1}) P = P
    \end{align*}
    for $f\in\Phi$, $g\in\Psi$. Since $\lambda(h\circ f\circ h^{-1}) = \lambda(f)$ and $O(h\circ f\circ h^{-1}) = O(f)$, similarly as in the proof of Corollary \ref{thm-3} we find that $P$ is invariant under
    \begin{align*}
    \langle \lbrace (\lambda(f)& O(f), I),\ (\lambda(f)^{-1} O(f)^{-1} , I), \\
    &(I, \lambda(g) O(g)),\ (I, \lambda(g)^{-1} O(g)^{-1}):\ f\in\Phi, g\in\Psi\rbrace\rangle.
    \end{align*}
    In particular, if 
    \begin{align*}
    \Gamma_\Phi &:= \langle \lbrace \lambda(f) O(f), \lambda(f)^{-1} O(f)^{-1}:\ f\in \Phi\rbrace\rangle, \\
    \Gamma_\Psi &:=\langle \lbrace \lambda(g) O(g), \lambda(g)^{-1} O(g)^{-1}:\ g\in \Psi\rbrace\rangle,
    \end{align*}
    then 
    \begin{align}\label{eq-integral}
        \dim \pi P &= \int \dim \pi(\eta\times\gamma)\,dP(\eta\times\gamma) \nonumber\\
        &= \int \dim (t O \eta *  s R\gamma)\,dP(\eta\times\gamma) \nonumber\\
        &= \int \dim (t s^{-1} OR^{-1} \eta *\gamma) \,dP(\eta\times\gamma)
    \end{align}
    for every $t O \in \Gamma_\Phi$, $sR \in \Gamma_\Psi$. The set
    \begin{align*}
    \Gamma &:= \lbrace ts^{-1} OR^{-1}:\ t O \in\Gamma_\Phi, sR\in \Gamma_\Psi\rbrace \\
    &= \Gamma_\Phi\Gamma_\Psi.
    \end{align*}
    satisfies the assumptions of Proposition \ref{prop-generationRd} by the assumptions (1)-(3) in the statement of the corollary, so it follows from Proposition \ref{prop-generationRd} that $\Gamma = CO(d)$. Therefore we may integrate \eqref{eq-integral} over all $ts^{-1} \geq 0$ and $OR^{-1} \in SO(d)$ to obtain, similarly as in the proof of Proposition \ref{prop-FDprojections}, that $\dim \pi P = \min\lbrace d,\dim P\rbrace$. Since $P$ was an arbitrary ergodic component of an arbitrary tangent distribution, it follows from Theorem \ref{thm-hochmanprojections} that $\dim (\mu*\nu) \geq \min\lbrace d, \dim P \rbrace = \min\lbrace d,\dim\mu+\dim\nu\rbrace$. 
\end{proof}

\begin{proof}[Proof of Corollary \ref{thm-7}]
    The proof proceeds exactly as the proof of Corollary \ref{thm-6} until up to the point where it is shown that $\Gamma = CO(d)$ is defined, so we only explain the simple modifications needed to show this equality. 
    
    Let $f_1, f_2, f_3, f_4$ be as in the assumption. Then the set
    \begin{equation*}
    \Gamma = \Gamma_\Phi\Gamma_\Psi.
    \end{equation*}
    contains the elements $\lambda(f_1) \lambda(f_2)^{-1} I$ and $\lambda(f_3)\lambda(f_4)^{-1}I$. Since $\Gamma$ is closed, it follows from the assumption (2) that $rI \in \Gamma $ for every $r>0$. The assumption (1) then implies that $\Gamma = CO(2)$, and the proof concludes exactly as before.
\end{proof}

It remains to prove the results on ortohogonal projections.

\begin{proof}[Proof of Theorem \ref{thm-8}]
    Fix a point $x \in \R^d$ in the support of $\mu$, and let $Q$ be a tangent distribution of $\mu$ at $x$ with ergodic decomposition $Q = \int P\,dQ'(P)$. Projecting the statement of Proposition \ref{prop-tangentdistributioninvariant} onto the first marginal and altering the choice of $x$ in a set of full measure, we find that for $Q'$-almost every $P$, there exists a conformal map $h:U\to\R^d$ such that 
    \begin{equation*}
        S_t D_{f_h}(x) P = P
    \end{equation*}
    for every $f\in \Phi$ and $h\mu$-almost every $x$, where $f_h = h\circ f\circ h^{-1}$. By total non-linearity, there exists $f\in\Phi$ such that $f_h$ is not affine. By possibly replacing $f_h$ with $f_h\circ f_h$, we assume that $f_h$ preserves orientation. By Proposition \ref{prop-keyprop} and the assumption that $\mu$ is not supported on a smooth hypersurface,
    \begin{equation*}
        \langle \lbrace \Vert D_{f_h}(x)\Vert^{-1} D_{f_h}(x):\ x \in h(K)\rbrace\rangle = SO(d).
    \end{equation*}
    In particular, $P$ is invariant under $SO(d)$. By \cite[Proposition 1.38]{Hochmanpreprint}, $\dim \pi P = \min \lbrace k, \dim P\rbrace$ for every orthogonal projection $\pi:\R^d\to \R^k$. Since $\dim P = \dim \mu$ by tangent-regularity of $\mu$, an application of Theorem \ref{thm-hochmanprojections} completes the proof.
\end{proof}

\begin{proof}[Proof of Theorem \ref{thm-projectionsalgebraic}]
    Applying Proposition \ref{prop-tangentdistributioninvariant}, we find for $\mu$-almost every $x\in\R^d$, any tangent distribution $Q$ at $x$ and almost every ergodic component $P$ of $Q$, a conformal map $h:U\to \R^d$ such that 
    \begin{equation*}
        S_t D{f_h}(x) P = P
    \end{equation*}
    for every $f\in\Phi$. Arguing as in the proof of Theorem \ref{thm-3} and using the assumption \eqref{eq-bestcondition} we find that $P$ is invariant under $SO(d)$. Therefore \cite[Proposition 1.38]{Hochmanpreprint} asserts that $\dim \pi P = \min \lbrace k, \dim P\rbrace$ for every orthogonal projection $\pi:\R^d\to \R^k$. Since $\dim P = \dim \mu$ by tangent-regularity of $\mu$, an application of Theorem \ref{thm-hochmanprojections} completes the proof.
\end{proof}

\begin{proof}[Proof of Theorem \ref{thm-plane}]
    If $\Phi$ is totally non-linear, then the statement follows from Theorem \ref{thm-8}. If $\mu$ gives positive measure to an analytic curve, then an application of the Lebesgue-Besicovitch density point theorem shows that $\mu$ is in fact supported on an analytic curve. By assumption, this curve is not a line, so any orthogonal projection $\pi:\R^2\to \R$ is locally bi-Lipschitz $\mu$-almost everywhere and the claim is trivial. So suppose that $\mu$ does not give positive measure to any analytic curve, and that there exists a conformal map $h$ such that $\Phi_h :=\lbrace h\circ f\circ h^{-1}:\ f\in\Phi\rbrace$ is self-similar. Since $\Phi$ contains a map which is not affine, the map $h$ cannot be affine. Then the measure $\nu:= h\mu$ is self-similar, and our aim is to show that for any orthogonal projection $\pi:\R^2\to\R$, 
    \begin{equation*}
        \dim \pi h^{-1} \nu =\min\lbrace 1,\dim\nu\rbrace.
    \end{equation*}
    
    By \cite[Theorem 1.1]{BaranyKaenmakiPyoralaWu2023}, the measure $\nu$ is uniformly scaling, i.e. it generates a unique ergodic fractal distribution $P$ at almost every point. By \cite[Theorem 1.15]{Hochmanpreprint}, there exists an ergodic \emph{CP-distribution} (see \cite[Definition 1.12]{Hochmanpreprint}) $Q$ such that
    \begin{equation}\label{eq-centering}
        P = \int_0^1 \int \delta_{\nu_{x,t}}\,dQ(\nu, x)\,dt
    \end{equation}
    Let $\Pi^{2,1}$ denote the set of orthogonal projections $\R^2\to\R$. By Wu's projection theorem for ergodic CP-distributions \cite[Theorem 1.5]{Wu2025}, there exists a countable set $E \subset \Pi^{2,1}$ such that for any $\pi \in \Pi^{2,1}\setminus E$ and $Q$-almost every $\nu$, $\dim \pi \nu = \min\lbrace 1,\dim Q\rbrace = \min\lbrace 1,\int\dim\nu \,dQ(\nu)\rbrace$. Since $\dim \pi\nu_{x,t}\geq \dim\pi\nu$ and the value of $\dim\pi\nu$ is $P$-almost surely constant by Theorem \ref{thm-hochmanrigidity}, it follows from \eqref{eq-centering} that
    \begin{equation}\label{eq:typicalprojections}
        \dim \pi P = \min\lbrace 1,\dim P\rbrace
    \end{equation}
    for every $\pi \not \in E$. By Proposition \ref{prop-imagedistribution}, for $h^{-1}\nu$-almost every $x$ and every large enough $t\geq 0$, 
    \begin{equation}\label{eq:hnugeneration}
        \lim_{T\to \infty} \frac{1}{T}\int_1^T \delta_{(h^{-1}\nu)_{x,t}}\,dt = S_t (Dh^{-1})(x) P.
    \end{equation}
    Fix now any $\pi \in G(2,1)$. Since $d=2$, $h^{-1}$ is complex analytic and $(Dh^{-1})(x) \in \C$, in particular, $\Vert (Dh^{-1})(x)\Vert^{-1}(Dh^{-1})(x)$ is a rotation by ${\rm Arg}((Dh^{-1})(x))$. We claim that for $h^{-1}\nu$-almost every $x$, $\pi \circ \Vert (Dh^{-1})(x)\Vert^{-1}(Dh^{-1})(x) \not\in E$. For a contradiction, suppose that there exists a set $A \subset \spt h^{-1}\nu$ with $h^{-1}\nu(A) >0$ such that for every $x\in A$, $\pi \circ \Vert (Dh^{-1})(x)\Vert^{-1}(Dh^{-1})(x) \in E$. It follows that there exists a countable set $E'\subset S^1$ such that $\Vert (Dh^{-1})(x)\Vert^{-1}(Dh^{-1})(x)\in E'$ for every $x\in A$, in particular, $(Dh^{-1})(A) \subset \mathcal L\subset \C$, where $\mathcal L$ is a countable union of lines passing through the origin. It follows that $A \subset (Dh^{-1})^{-1}(\mathcal L)$, a countable union of analytic curves, so by countable stability $h^{-1} \nu$ gives positive measure to an analytic curve. This contradicts the assumption made in the beginning of the proof. 
    
    We conclude that $\pi \circ \Vert (Dh^{-1})(x)\Vert^{-1}(Dh^{-1})(x) \not\in E$ for $h^{-1}\nu$-almost every $x$. In particular, by \eqref{eq:typicalprojections} and $S_t$-invariance of $P$,
    \begin{equation*}
        \dim \pi S_{t}(Dh^{-1})(x) P = \min\lbrace 1,\dim P\rbrace
    \end{equation*}
    for $h^{-1}\nu$-almost every $x$ and every large enough $t\geq 0$. By \eqref{eq:hnugeneration}, Proposition \ref{prop-localdimensions} and Theorem \ref{thm-hochmanprojections}, $\dim \pi\mu = \dim \pi h^{-1}\nu = \min\{1, \dim P\} = \min\{1, \dim\mu\}$. 
\end{proof}
    One could likely make the proof of Theorem \ref{thm-plane} work with only the assumption $\dimh E = 0$. In particular, one could likely invoke Hochman's result on exceptional projections of self-similar measures \cite[Theorem 1.8]{Hochman2014} in place of \cite[Theorem 1.5]{Wu2025}, but it would require translating the theorem to the language of CP-distributions generated by self-similar measures.

\bibliographystyle{abbrv}
\bibliography{Bibliography_old}
\end{document}